\documentclass[a4paper,twoside]{article}
\usepackage{a4}
\usepackage{amssymb}
\usepackage{amsmath}
\usepackage{amsthm}
\usepackage{mathtools}
\usepackage{graphicx}
\usepackage{pgf, tikz}
\usepackage{float}
\usepackage{subcaption}
\usetikzlibrary{shapes}
\usetikzlibrary{arrows, automata, patterns}

\usepackage{authblk}

\newtheorem{theorem}{Theorem}[section]
\newtheorem{lemma}[theorem]{Lemma}

\newtheorem{proposition}[theorem]{Proposition}

\newtheorem{definition}{Definition}[section]
\newtheorem{remark}{Remark}[section]

\numberwithin{equation}{section}

\DeclareMathOperator*{\divergence}{div}

\DeclareMathOperator*{\argmin}{arg\,min}
\DeclareMathOperator*{\sgn}{sgn}

\title{Invariant $\varphi$-minimal sets and total variation denoising on graphs}
\author[1]{Clemens Kirisits}
\author[1]{Eric Setterqvist}
\author[1,2]{Otmar Scherzer}
\affil[1]{Faculty of Mathematics, University of Vienna, Vienna, Austria}
\affil[2]{Johann Radon Institute for Computational and Applied Mathematics (RICAM), Austrian Academy of Sciences, Linz, Austria}
\date{July 31, 2019}

\begin{document}

\maketitle

\begin{abstract}
Total variation flow, total variation regularization and the taut string algorithm are known to be equivalent filters for one-dimensional discrete signals. 
In addition, the filtered signal simultaneously minimizes a large number of convex functionals in a certain neighbourhood of the data.
In this article we study the question to what extent this situation remains true in a more general setting, namely for data given on the vertices of an oriented graph and the total variation being $J(f) = \sum_{i,j} |f(v_i) - f(v_j)|$.
Relying on recent results on invariant $\varphi$-minimal sets we prove that the minimizer to the corresponding Rudin-Osher-Fatemi (ROF) model on the graph has the same universal minimality property as in the one-dimensional setting.
Interestingly, this property is lost, if $J$ is replaced by the discrete isotropic total variation.
Next, we relate the ROF minimizer to the solution of the gradient flow for $J$.
It turns out that, in contrast to the one-dimensional setting, these two problems are not equivalent in general, but conditions for equivalence are available.
\end{abstract}

\section{Introduction}\label{sec:intro}
It is a well known fact that for one-dimensional discrete data total variation (TV) regularization and TV flow are equivalent.
More specifically, denote by
	$$ J(u) =  \sum_{i=1}^{n-1} |u_i - u_{i+1}| $$
the total variation of $u \in \mathbb{R}^n$, and let $f\in \mathbb{R}^n$ and $\alpha>0$ be given.
Then, as was shown in \cite{Steidl1}, the minimizer $u_\alpha$ of the functional
	$$ \frac{1}{2} \|f - u\|^2_{2} + \alpha J(u)$$
coincides with the solution to the Cauchy problem
	\begin{align*}
		u'(t) &\in - \partial J(u(t)), \quad t>0,\\
		u(0)  &= f,
	\end{align*}
at time $t = \alpha$. That is, $u_\alpha = u(\alpha)$ for all $\alpha >0.$
On the other hand, it is known that $u_\alpha$ can also be obtained by means of the taut string algorithm (see \cite{Mammen}), which reads as follows.
\begin{enumerate}
	\item Identify the vector $f\in \mathbb{R}^n$ with a piecewise constant function on the unit interval and integrate it to obtain the linear spline $F$.
	\item Find the ``taut string" $U_\alpha$, that is, the element of minimal graph length in a tube of width $2\alpha$ around $F$ with fixed ends:
		$$ U_\alpha = \argmin \left\{ \int_0^1 \sqrt{1 + (U'(x))^2}\, dx: \|U - F\|_\infty \le \alpha, U(0)=F(0), U(1) = F(1) \right\} $$
	\item Differentiate $U_\alpha$ to obtain $u_\alpha.$
\end{enumerate}
Problems which essentially can be modelled and solved by the taut string algorithm appear in diverse applications. Examples include production planning, see for instance \cite{Modigliani1}, and energy and information transmission, e.g.\ \cite{Salehi1} and \cite{Yang1}.
Extensions of the taut string algorithm to more general data have been studied in \cite{Gra07,GraObe08,Hinterberger1}.
Further suggestions of generalizations of the taut string algorithm, in both discrete and continuous settings, can be found in \cite[Chap. 4.4]{Scherzer1}.

It turns out that the taut string does not only have minimal graph length, but actually minimizes \emph{every} functional of the form
\begin{equation*}
	U \mapsto \int^1_0 \varphi(U'(x))\, dx,
\end{equation*}
where $\varphi:\mathbb{R} \to \mathbb{R}$ is an arbitrary convex function and $U$ ranges over the $2\alpha$-tube around $F$. Recently, this intriguing situation was studied in greater generality in \cite{Kruglyak2,Kruglyak3}.
The authors coined the term \emph{invariant $\varphi$-minimal} for sets which, like the $2\alpha$-tube, have an element that simultaneously minimizes a large class of distances. In addition they characterized these sets in the discrete setting.

In this article we study relations between TV regularization, TV flow and taut strings in a setting that contains the one outlined above as a special case. More specifically, we consider data $f$ as given on the vertices of an oriented graph $G=(V,E)$ together with the total variation
\begin{equation} \label{eq:tv}
	J(f) = \sum_{v,w} |f(v) - f(w)|,
\end{equation}
where the sum runs over all adjacent pairs of vertices $v,w$.

Our first result concerns the subdifferential of $J$. In Theorem \ref{thm:subdifphimin} we prove that $\partial J(f)$ is an invariant $\varphi$-minimal set for every $f:V\to \mathbb{R}$. It is noteworthy that, as is shown in Remark \ref{rem:isorof}, this property is not shared by the discrete isotropic total variation, which for $f\in \mathbb{R}^{m \times n}$ reads\footnote{Here, $f\in \mathbb{R}^{m \times n}$ corresponds to $f$ being defined on the vertices of an $m\times n$ Cartesian graph as depicted in Figure \ref{F1}.}
\begin{equation}\label{eq:isotv}
	\sum_{i,j }\sqrt{(f_{i+1,j} - f_{i,j})^{2} + (f_{i,j+1}-f_{i,j})^{2}}
\end{equation}
and has been widely used in imaging applications, see \cite{AujAubBlaCha05,AujGilChaOsh06,Chambolle2} for instance.

Next we consider the Rudin-Osher-Fatemi (ROF) model \cite{Rudin1} on the graph
\begin{equation} \label{eq:rofgraph}
	\min_{u : V \to \mathbb{R}} \frac{1}{2} \sum_{v\in V}|f(v) - u(v)|^2 + \alpha J(u), \quad \alpha \ge 0.
\end{equation}
From its dual formulation and Theorem \ref{thm:subdifphimin} it follows that the solution $u_\alpha$ of problem \eqref{eq:rofgraph} has a characteristic feature resembling the universal minimality property of the taut string: It simultaneously minimizes
\begin{equation*}
	 \sum_{v\in V} \varphi (u(v))
\end{equation*}
over the set $f - \alpha \partial J(0)$ for every convex $\varphi,$ see Theorem \ref{T1}. We stress again that the minimizer of the isotropic ROF model, where $J(f)$ is given by \eqref{eq:isotv}, does not have this property.

Because of its anisotropy different variants of model \eqref{eq:rofgraph} have been used for imaging problems with an underlying rectilinear geometry \cite{BerBurDroNem06,choksi2011,SanOzkRomGok18,Setzer1}. Moreover, in contrast to \eqref{eq:isotv}, $J$ as given by \eqref{eq:tv} is submodular and for the minimization of submodular functions many efficient algorithms are available, for instance, graph cut algorithms \cite{Cha05,ChaDar09,DarSig06a,Hoc01}.

Finally, we examine the gradient flow for $J$ and how it relates to the ROF model. Such relations in higher dimensional settings have been the subject of recent investigations. In \cite{Burger1} discrete variational methods and gradient flows for convex one-homogeneous functionals are investigated and sufficient conditions for their equivalence are provided. A sufficient condition for the equivalence of TV regularization and TV flow with $\ell^{1}$-anisotropy in the continuous two-dimensional setting is given in \cite{Lasica1}. Considering the continuous setting with isotropic TV, it is shown in \cite{Jalalzai1} that TV regularization and TV flow coincide for radial data but in general are non-equivalent.

Our results in this direction are the following. First and foremost TV regularization and TV flow are not equivalent for general graphs and data $f$, see Theorem \ref{prop:noneq}. This result is based on a constructed example for which we are able to explicitly track the evolution of the two solutions $u_\alpha$ and $u(t)$ as $\alpha$ and $t$ range over an interval $[0,L].$ The example also shows that, in contrast to the one-dimensional setting, the jump sets do not necessarily evolve in a monotone way. Moreover, we investigate conditions for equality of $u_\alpha$ and $u(t=\alpha)$ and discuss situations in which they apply. 

To summarize, let $\psi:\mathbb{R}\to\mathbb{R}$ by a strictly convex function, for the sake of analogy pick $\psi(x) = \sqrt{1 + x^2}.$ Then the problem
\begin{equation*}
	\min_{u \in f - \alpha \partial J(0)} \sum_{v\in V} \psi (u(v))
\end{equation*}
may be seen as a generalization of the taut string algorithm to oriented graphs for the following reasons.
\begin{itemize}
	\item The set $f - \alpha\partial J(0)$ reduces to the set of derivatives of the elements in the $2\alpha$-tube around $F$ in case the underlying graph is a path, that is, it models the one-dimensional situation described in the first paragraph of this introduction.
	\item The solution $u_{\alpha}$ in fact minimizes $\sum_{v\in V}\varphi(u(v))$ for any convex function $\varphi$.
	\item $u_{\alpha}$ minimizes the corresponding ROF model \eqref{eq:rofgraph}.
	\item Further, if $\alpha$ is either sufficiently small or sufficiently large, then $u_{\alpha}$ equals the TV flow solution at time $t=\alpha$.
\end{itemize}

This article is organized as follows. In Section \ref{S1} we introduce the graph setting and collect some properties of the total variation $J$. In particular we discuss the concept of invariant $\varphi$-minimal sets in Section \ref{sec:phiminimal}, while establishing a connection to base polyhedra in Section \ref{sec:submodular}. Sections \ref{sec:rof} and \ref{sec:tvflow} are dedicated to the two main problems considered in this paper, that is, total variation regularization and total variation flow, respectively. In Section \ref{sec:comparison} we compare the flow and ROF solutions. The detailed calculations underlying several results of Section \ref{sec:comparison} are collected in the appendix.

\section{Total variation on graphs} \label{S1}
Throughout this article, following the terminology of \cite{ChaLesZha10}, we consider oriented connected graphs $G=(V,E)$. That is, $V=\{v_1,\ldots,v_n\}$ and $E \subset V \times V$ with the additional conditions that, first, $(v_i,v_j) \in E$ implies $(v_j,v_i) \notin E$ and, second, there is a path between every pair of vertices (ignoring edge orientations). Whenever we simply write ``graph" below, we implicitly mean a graph of this type.
For $v,w \in V$ the edge $(v,w)\in E$ is interpreted as directed from $v$ to $w$.
Let $\mathbb{R}^{V}$ and $\mathbb{R}^{E}$ be the space of real-valued functions defined on the vertices and edges, respectively. We consider the usual $\ell^{p}$-norms on $\mathbb{R}^{V}$
\begin{align*}
	\|u\|^p_{p}		&=	\sum_{v\in V}|u(v)|^{p}, \quad 1\leq p<\infty, \\
	\|u\|_{\infty}	&=	\max_{v\in V}|u(v)|.
\end{align*}
Analogous $\ell^{p}$-norms will be considered on $\mathbb{R}^{E}$. In particular, denote the closed $\ell^{\infty}$-ball of radius $\alpha \ge 0$ in $\mathbb{R}^{E}$ by
\begin{align*}
 \mathcal{B}_{\alpha}=\{H\in \mathbb{R}^{E}:\lVert H\rVert_{\infty}\leq\alpha\}.
\end{align*}
Given $H\in \mathbb{R}^{E}$, define the \emph{divergence operator} $\divergence:\mathbb{R}^{E}\rightarrow \mathbb{R}^{V}$ according to
\begin{align*}
 (\divergence H)(v)=\sum_{w\in V:(w,v)\in E}H((w,v))-\sum_{w\in V:(v,w)\in E}H((v,w)).
\end{align*}
The divergence at the vertex $v$ can be thought of as the sum of the flows on the incoming edges minus the sum of the flows on the outgoing edges. We will frequently apply $\divergence$ to the unit ball $\mathcal{B}_{1}\in\mathbb{R}^{E}$ and its subset $\mathcal{B}_{1,u}$ defined, for given $u\in\mathbb{R}^{V}$, by
\begin{align*} 
\mathcal{B}_{1,u}=\left\{H\in \mathbb{R}^{E}:
 H((v_{i},v_{j}))\in
 \left\{
\begin{array}{ll}
\{1\}, & u(v_{i})<u(v_{j}),\\
\left[-1,1\right], & u(v_{i})=u(v_{j}),\\
\{-1\}, & u(v_{i})>u(v_{j})
\end{array}
\right.
\right\}.
\end{align*}
Introduce further the natural scalar product on $\mathbb{R}^{V}$ according to
\begin{align*}
 \langle u,h\rangle_{\mathbb{R}^{V}}=\sum_{v\in V}u(v)h(v).
\end{align*}
For a closed and convex set $A\subset \mathbb{R}^V$ the \emph{support function} $\sigma_A : \mathbb{R}^V \to \mathbb{R}$ is given by
\begin{align*}
	\sigma_A (u) = \sup_{h\in A}\langle u , h \rangle_{\mathbb{R}^{V}}.
\end{align*}
\begin{definition}\label{def:tv}
The \emph{total variation} on $\mathbb{R}^{V}$ is defined as the support function of the set $\divergence \mathcal{B}_1$,
\begin{align}
 J(u)	&=	\underset{h\in\divergence \mathcal{B}_{1}}{\sup}\langle u,h\rangle_{\mathbb{R}^{V}}. \notag
\intertext{Since $J(u) = \langle u, \divergence H \rangle_{\mathbb{R}^V}$ for every $H \in \mathcal{B}_{1,u}$ we can rearrange the inner product to obtain} 
 J(u)	&=	\sum_{(v_{i},v_{j})\in E}|u(v_{j})-u(v_{i})|. \label{TVG2}
\end{align}
\end{definition}

\begin{remark}
Equation \eqref{TVG2} shows that $J$ is independent of the orientation of edges, even though the divergence is not. All subsequent results remain true regardless of edge orientation, and also apply to simple undirected graphs once each edge has been oriented arbitrarily.
\end{remark}

\begin{definition}\label{dfn:subdif}
For every $u\in \mathbb{R}^V$ the \emph{subdifferential} $\partial J(u)$ is defined as the set of all elements $u^{*}\in \mathbb{R}^{V}$ such that
\begin{align*}
\langle h-u,u^{*}\rangle_{\mathbb{R}^{V}}+J(u)\leq J(h) \quad \text{for all } h\in \mathbb{R}^{V}.
\end{align*}
\end{definition}
Since $\partial J(u)$ is a closed, convex and non-empty subset of $\mathbb{R}^{V}$, we can highlight one particular element.
\begin{definition}
 The element of minimal $\ell^2$-norm in $\partial J(u)$ will be referred to as the \emph{minimal section} of $\partial J(u)$. It is denoted by $\partial^\circ J(u)$, that is,
 \begin{align*}
  \partial^\circ J(u) = \underset{u^{*}\in\partial J(u)}{\argmin}\lVert u^{*}\rVert_{{2}}.
 \end{align*}
\end{definition}
\par The following lemma collects some results for the subdifferential $\partial J$ which will be used in the sequel.
\begin{lemma} \label{L1}
\leavevmode
\begin{enumerate}
\item $\partial J(0)=\divergence\mathcal{B}_{1}$.
\item $\partial J(u)=\{u^{*}\in\partial J(0):\langle u,u^{*}\rangle_{\mathbb{R}^{V}}=J(u)\} \quad \text{for all } u\in \mathbb{R}^{V}$.
\item $\partial J(u)=\divergence\mathcal{B}_{1,u} \quad \text{for all } u\in \mathbb{R}^{V}$.
\end{enumerate}
\end{lemma}
\begin{proof}
The functional $J$ is the support function of the closed and convex set $\divergence\mathcal{B}_{1}$ and therefore $\partial J(0)=\divergence\mathcal{B}_{1}$.

Item 2 follows from Definition \ref{dfn:subdif} and the absolute $1$-homogeneity of $J$, that is, $J(tu) = |t|J(u)$ for all $t\in \mathbb{R}$ and $u \in \mathbb{R}^V$.

Regarding item 3, note that $J(u)= \langle u,\divergence H \rangle_{\mathbb{R}^V}$ for $H\in\mathcal{B}_{1}$ if and only if $H\in\mathcal{B}_{1,u}$. In view of item 2, it is then clear that $\partial J(u)=\divergence\mathcal{B}_{1,u}$.
\end{proof}
\begin{remark}\label{rem:subdifequal1}
\leavevmode
\begin{enumerate}
\item Since, according to item 3 in Lemma \ref{L1}, the set $\mathcal{B}_{1,u}$ only depends on $\sgn(u(v_i) - u(v_j))$ for every edge $(v_{i},v_{j})\in E$, we have 
\begin{align*}
 \partial J(u)=\partial J(h),
\end{align*}
if and only if
\begin{align*}
 \sgn(u(v_{i})-u(v_{j}))=\sgn(h(v_{i})-h(v_{j}))
\end{align*}
for each $(v_{i},v_{j})\in E$.
\item It now follows immediately that, if the subdifferentials of $J$ at $u$ and $h$ coincide, then they also coincide for every convex combination of $u$ and $h$. That is, $\partial J(u) = \partial J(h)$ implies $\partial J(\lambda u + (1-\lambda)h) = \partial J(u)$ for every $\lambda \in (0,1)$.
\item Lemma \ref{L1} also implies that the number of different subdifferentials of $J$ is finite. In particular,
 \begin{align*}
  \left|\left\{\partial J(u):u\in \mathbb{R}^{V}\right\}\right|\leq 3^{\left|E\right|}.
 \end{align*}
 This must not be confused with the fact that for any given $u\in \mathbb{R}^{V}$ the subdifferential $\partial J(u)$ might have infinitely many elements.
\end{enumerate}
\end{remark}

\subsection{Connections to invariant $\varphi$-minimal sets}\label{sec:phiminimal}
In this subsection we recall the notion of invariant $\varphi$-minimal sets introduced in \cite{Kruglyak2} and show that the subdifferential $\partial J(u)$ is an example of such a set.
\begin{definition}\label{def:invphimin}
 A set $\Omega\subset\mathbb{R}^{n}$ is called \emph{invariant $\varphi$-minimal} if for every $a\in\mathbb{R}^{n}$ there exists an element $x_{a}\in\Omega$ such that
 \begin{align} \label{IP}
  \sum^{n}_{i=1}\varphi(x_{a,i}-a_{i})\leq\sum^{n}_{i=1}\varphi(x_{i}-a_{i})
 \end{align}
holds for all $x\in\Omega$ and all convex functions $\varphi:\mathbb{R}\rightarrow\mathbb{R}$.
\end{definition}
An interesting property of invariant $\varphi$-minimal sets is the following. By considering the particular convex function $\varphi(x)=\left|x\right|^{p}$, $1\leq p<\infty$, in \eqref{IP} we obtain
\begin{align*}
\sum^{n}_{i=1}\left|x_{a,i}-a_i\right|^{p} \leq \sum^{n}_{i=1} \left|x_i-a_i\right|^{p}
\end{align*}
for all $x\in\Omega$. Taking the $p$-th root and including the case $p=\infty$, which follows by limiting arguments, shows that the element $x_{a}$ satisfies
\begin{align*}
 \lVert x_{a}-a\rVert_{{p}}\leq\lVert x-a\rVert_{{p}}
\end{align*}
for all $x\in\Omega$ and $1\leq p\leq\infty$. That is, $x_{a}$ is an element of best approximation of $a$ in $\Omega$ with respect to all $\ell^{p}$-norms, $1\leq p\leq\infty$.

Before we can restate two characterizations of invariant $\varphi$-minimal sets from \cite{Kruglyak2} we have to introduce several notions about convex subsets of $\mathbb{R}^n$.

A hyperplane $H$ \emph{supports} a set $M\subset\mathbb{R}^{n}$ if $M$ is contained in one of the two closed halfspaces with boundary $H$ and at least one boundary point of $M$ is in $H$. Assume that $M\subset\mathbb{R}^{n}$ is convex. Following the terminology of \cite{Grunbaum2} a set $F\subset M$ is called a \emph{face} of $M$ if $F=\emptyset$, $F=M$ or if $F=M\cap H$ where $H$ is a supporting hyperplane of $M$. A \emph{convex polytope} $P$ in $\mathbb{R}^{n}$ is a bounded set which is the intersection of finitely many closed halfspaces. Note that a face of a convex polytope is itself a convex polytope.

 Let $\Omega\subset\mathbb{R}^{n}$ be closed and convex and denote by $\left\{e_{i}\right\}^{n}_{i=1}$ the standard basis of $\mathbb{R}^{n}$. For $x\in\Omega$, consider all vectors $y=e_{i}-e_{j}$ such that $x+\beta y\in\Omega$ for some $\beta>0$. Let $S_{x}$ denote the set of all such vectors at $x$. Further, let $K_{x}=\{z:z=\sum_{y\in S_{x}}\lambda_{y}y,\lambda_{y}\geq 0\}$ be the convex cone generated by the vectors in $S_{x}$. We say that $\Omega$ has the \emph{special cone property} if $\Omega\subset x+K_{x}$ for each $x\in\Omega$.
\begin{remark}
In \cite{Kruglyak2} vectors of the type $e_{i}$ and $e_{i}+e_{j}$ are considered in addition to $e_i-e_j$ in the definition of the special cone property. Including these vectors leads to a characterization of the related notion of invariant $K$-minimal sets.
\end{remark}

\begin{theorem}\label{thm:phimin}
Let $\Omega\subset\mathbb{R}^{n}$ be a bounded, closed and convex set. Then the following statements are equivalent.
\begin{enumerate}
	\item $\Omega$ is invariant $\varphi$-minimal.
	\item $\Omega$ has the special cone property.
	\item $\Omega$ is a convex polytope where the affine hull of any of its faces is a shifted subspace of $\mathbb{R}^{n}$ spanned by vectors of the type $e_{i}-e_{j}$.
\end{enumerate}
\end{theorem}
\begin{proof}
Equivalence of statements 1 and 2 follows from combining Thms.\ 3.2 and 4.2 in \cite{Kruglyak2}. Equivalence of statements 1 and 3 is precisely Thm.\ 4.3 in \cite{Kruglyak2}.
\end{proof}

An example of an invariant $\varphi$-minimal set in the plane is depicted in Figure \ref{fig:invphimin}, left panel.
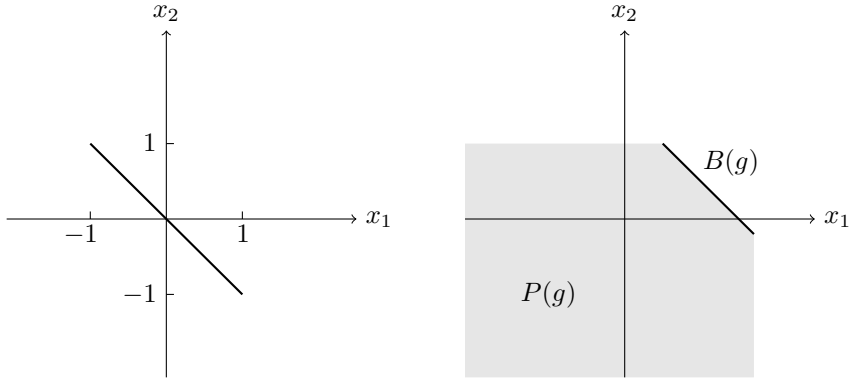
\begin{figure} 
\centering
\begin{tikzpicture}[scale=1]
    \draw[->] (-2.1,0)  -- (2.5,0) node(xline)[right] {$x_1$};
    \draw[->] (0,-2.1) -- (0,2.5) node(yline)[above] {$x_2$};
    \draw[-,thick] (-1,1) -- (1,-1);
    \draw[-] (1,0) -- (1,.1);
    \draw[-] (0,1) -- (.1,1);
    \draw[-] (-1,0) -- (-1,.1);
    \draw[-] (0,-1) -- (.1,-1);
    \node at (.875,-.2)  {$\phantom{-}1$};
    \node at (-.35,1)  {$\phantom{-}1$};
    \node at (-1.125,-.2)  {$-1$};
    \node at (-.35,-1)  {$-1$};
\end{tikzpicture}
\qquad
\begin{tikzpicture}[scale=1]
    \fill[gray!20] (-2,1) -- (0.5,1) -- (1.7,-0.2) -- (1.7,-2.1) --(-2.1,-2.1) -- (-2.1,1);
    \draw[->] (-2.1,0)  -- (2.5,0) node(xline)[right] {$x_{1}$};
    \draw[->] (0,-2.1) -- (0,2.5) node(yline)[above] {$x_{2}$};
    \draw[-,thick] (0.5,1) -- (1.7,-0.2);
    \node at (-1,-1)  {$P(g)$};
    \node at (1.4,0.75)  {$B(g)$};
\end{tikzpicture}
\caption{Left: The slanted line segment is an example of an invariant $\varphi$-minimal set as characterized by Theorem \ref{thm:phimin}. In fact, identifying $x_i$ with $u^*(v_i)$, it is the subdifferential $\partial J(0)$ for $J$ being defined on the graph with $V=\{v_1,v_2\}$ and $E=\{(v_1,v_2)\}$. All other invariant $\varphi$-minimal sets in $\mathbb{R}^2$ are translations and rescalings of $\partial J(0)$. Right: Submodular polyhedron (in grey) and base polyhedron (slanted line segment) in the plane.}
\label{fig:invphimin}
\end{figure}
We are now ready to show
\begin{theorem} \label{thm:subdifphimin}
 The subdifferential $\partial J(u)$ is an invariant $\varphi$-minimal set.
\end{theorem}
\begin{proof}
Consider first $\partial J(0)$. In \cite[Thm.\ 2.4, Rem.\ 2.5]{Kruglyak3} it is established that the bounded, closed and convex set $\divergence\mathcal{B}_{\alpha}\subset\mathbb{R}^{V}$ is invariant $\varphi$-minimal by showing that it has the special cone property. It follows that $\partial J(0)=\divergence\mathcal{B}_{1}$ is an invariant $\varphi$-minimal set.
\par Take next a general $u\in \mathbb{R}^{V}$. We have $\partial J(u)=H\cap\partial J(0)$ where $H=\{u^{*}\in\mathbb{R}^{V}:\langle u^{*},u\rangle_{\mathbb{R}^{V}}=J(u)\}$, recall Lemma \ref{L1}. Consider the halfspace $\widehat{H}=\{u^{*}\in\mathbb{R}^{V}:\langle u^{*},u\rangle_{\mathbb{R}^{V}}\leq J(u)\}$ with boundary $H$. Note that (i) $\partial J(0)\subset\widehat{H}$, (ii) $H\cap\partial J(0)\neq\emptyset$ and (iii) $\partial J(0)$ is a convex polytope. So, $H$ is a supporting hyperplane of $\partial J(0)$ and $\partial J(u)$ is a face of $\partial J(0)$ and itself a convex polytope. Further, every face of $\partial J(u)$ is a face of $\partial J(0)$. This follows from a general result on faces of convex polytopes, see e.g. \cite[Chap.\ 3.1, Thm.\ 5]{Grunbaum2}. Therefore $\partial J(u)$ satisfies statement 3 in Theorem \ref{thm:phimin}.
\end{proof}
\begin{remark}
 As $\partial J(u)$ is an invariant $\varphi$-minimal set, it follows that the minimal section $\partial^\circ J(u)$ not only has minimal $\ell^2$-norm in $\partial J(u)$, but satisfies
 \begin{align*}
  \sum_{v\in V}\varphi(\partial^\circ J(u)(v))=\underset{u^{*}\in\partial J(u)}{\min}\sum_{v\in V}\varphi(u^{*}(v))
 \end{align*}
for every convex function $\varphi:\mathbb{R}\rightarrow\mathbb{R}$.
\end{remark}

\subsection{Invariant $\varphi$-minimal sets and submodular functions} \label{sec:submodular}
To conclude this section, we present an interesting connection between submodular functions and invariant $\varphi$-minimal sets. Submodular functions play an important role in combinatorial optimization, similar to that of convex functions in continuous optimization. See \cite{Bach13,Fujishige1} for more details.
\par Let $S=\left\{1,...,n\right\}$. A set function $g:2^{S}\rightarrow\mathbb{R}$ is \emph{submodular} if
 \begin{align*}
  g(A)+g(B)\geq g(A\cup B)+g(A\cap B)
 \end{align*}
for all sets $A,B\subset S$. Given a submodular function $g$, assuming $g(\emptyset)=0$, the associated \emph{submodular polyhedron} $P(g)$ and \emph{base polyhedron} $B(g)$ are defined by
\begin{align*}
 P(g)=\left\{x\in\mathbb{R}^{n}:\forall A\subset S,\sum_{i\in A}x_{i}\leq g(A)\right\},
\end{align*}
\begin{align*}
 B(g)=\left\{x\in P(g):\sum_{i\in S}x_{i}=g(S)\right\}.
\end{align*}
Note that $B(g)$ is a bounded set and therefore a convex polytope. In the plane we can easily visualize submodular and base polyhedra, see Figure \ref{fig:invphimin}, right panel, for an example.

\par Define the \emph{tangent cone} $T_{P}(x)$ of a convex polytope $P\subset\mathbb{R}^{n}$ at $x\in P$ by
\begin{align*}
 T_{P}(x)=\left\{\lambda z:\lambda\geq 0,x+z\in P\right\}.
\end{align*}
N. Tomizawa characterized, see \cite[Thm. 17.1]{Fujishige1}, base polyhedra according to
\begin{theorem} \label{thm:basepoly}
 A convex polytope $P\subset\mathbb{R}^{n}$ is a base polyhedron if and only if for all $x\in P$, the tangent cone $T_{P}(x)$ is generated by vectors of the type $e_{i}-e_{j}$, $i\neq j$.
\end{theorem}
With this characterization at hand, the connection between invariant $\varphi$-minimal sets and submodular functions can be revealed.
\begin{proposition} \label{thm:equivbasephi}
 A bounded, closed and convex set $\Omega\subset\mathbb{R}^{n}$ is invariant $\varphi$-minimal if and only if it is a base polyhedron associated to a submodular function $g:2^{S}\rightarrow\mathbb{R}$.
\end{proposition} 
\begin{proof}
 Recall from Theorem \ref{thm:phimin} that $\Omega$ is invariant $\varphi$-minimal if and only if it has the special cone property. Next, it is straightforward to derive that $\Omega$ has the special cone property if and only if the tangent cone $T_{\Omega}(x)$, for every $x\in\Omega$, is generated by vectors of the type $e_{i}-e_{j}$, $i\neq j$. This is precisely the characterization of a base polyhedron as given by Theorem \ref{thm:basepoly}.
\end{proof}
\begin{remark}
Figure \ref{fig:invphimin} illustrates the equivalence of invariant $\varphi$-minimal sets and base polyhedra. Note that the subdifferential $\partial J(0)$ is the base polyhedron $B(g)$ associated to the cut function $g$ on the graph, see \cite[Sec.\ 6.2]{Bach13}.
\end{remark}

\section{The ROF model on the graph}\label{sec:rof}
With the graph setting introduced, we now turn to an analogue of the ROF image denoising model on $\mathbb{R}^{V}$. Given $f\in \mathbb{R}^{V}$ and $\alpha \ge 0$ we consider the following minimization problem:
\begin{align} \label{ROFG}
 \underset{u\in \mathbb{R}^{V}}{\min} \frac{1}{2}\|f-u\|^{2}_{{2}}+\alpha J(u).
\end{align}
Throughout this article the unique solution to \eqref{ROFG} will be denoted by $u_\alpha$.

\subsection{Dual formulation and an invariance property of the ROF minimizer} \label{S31}
The next proposition remains true, if $J$ is replaced by the support function of an arbitrary closed and convex subset of $\mathbb{R}^V.$
\begin{proposition}\label{thm:rofconstrained}
For every $f\in \mathbb{R}^{V}$ and $\alpha \ge 0$ problem \eqref{ROFG} is equivalent to
\begin{align} \label{206}
\underset{u\in f - \alpha \partial J(0)}{\min}\|u\|_{{2}}.
\end{align}
\end{proposition}
\begin{proof}
The corresponding dual problem of \eqref{ROFG} can be expressed as
\begin{align} \label{ROFGD}
 \underset{u^*\in \mathbb{R}^{V}}{\min} \frac{1}{2}\|f-u^*\|^{2}_{{2}}+(\alpha J)^{*}(u^*),
\end{align}
where $(\alpha J)^{*}$ denotes the convex conjugate of $\alpha J$. For general results underlying the derivation of \eqref{ROFGD} and the optimality conditions \eqref{OC1} below, see \cite[Chap. III, Prop. 4.1, Rem. 4.2]{Ekeland2}. Let $u_{\alpha}$ and $u^*_{\alpha}$ denote solutions to the primal problem \eqref{ROFG} and the dual problem \eqref{ROFGD} respectively. The optimality conditions are
\begin{equation}\label{OC1}
	\begin{aligned}
		u^*_{\alpha}	&\in\partial(\alpha J)(u_{\alpha})=\alpha\partial J(u_{\alpha}) \\
		u_{\alpha}	&=f-u^*_{\alpha}.
	\end{aligned}
\end{equation}
As $\alpha J(u)$ is the support function of $\divergence\mathcal{B}_{\alpha} = \alpha \partial J(0)$, its convex conjugate $(\alpha J)^{*}$ is given by
\begin{equation*}
\begin{aligned}
(\alpha J)^{*}(u^*)=
\left\{
\begin{array}{lr}
0,			& u^* \in \alpha \partial J(0),\\
+\infty,	& u^* \notin \alpha \partial J(0).
\end{array}
\right.
\end{aligned}
\end{equation*}
Taking into account the characterization of $(\alpha J)^{*}$ in the dual formulation \eqref{ROFGD} yields
\begin{align*}
 u^*_{\alpha} = \underset{u^*\in\alpha \partial J(0)}{\argmin}\|f-u^*\|_{2}.
\end{align*}
That is, $u^*_{\alpha}$ is the orthogonal projection of $f$ onto the closed and convex set $\divergence\mathcal{B}_{\alpha}$. For $u_{\alpha}$ we now obtain using \eqref{OC1} that
\begin{align*}
	\left\|u_{\alpha}\right\|_{{2}}=\|f-u^*_{\alpha}\|_{{2}}
		= \underset{u^*\in\alpha \partial J(0)}{\min}\left\|f-u^*\right\|_{{2}}
		= \underset{u\in f-\alpha \partial J(0)}{\min}\left\|u\right\|_{{2}}.
\end{align*}
\end{proof}
\begin{theorem} \label{T1}
The ROF minimizer $u_{\alpha}$ satisfies
 \begin{align} \label{TVRP}
 \sum_{v\in V}\varphi(u_{\alpha}(v))=\underset{u\in f-\alpha\partial J(0)}{\min}\sum_{v\in V}\varphi(u(v))
\end{align}
for every convex function $\varphi:\mathbb{R}\rightarrow\mathbb{R}$.
\end{theorem}
\begin{proof}
According to Theorem \ref{thm:subdifphimin} the set $\alpha\partial J(0)$ is invariant $\varphi$-minimal, recall Definition \ref{def:invphimin}. From the above derivation of the dual formulation, we know that $u_{\alpha}$ is the $\ell^{2}$-minimizer in the set $f-\alpha\partial J(0)$. Taken together, this gives \eqref{TVRP}.
\end{proof}
While Proposition \ref{thm:rofconstrained} is valid for every support function of a closed and convex set, Theorem \ref{T1} fails in this more general case. The following remark discusses this failure for the so-called discrete isotropic total variation.
\begin{remark}\label{rem:isorof}
Let $G=(V,E)$ be an $M\times N$ Cartesian graph, as illustrated in Figure \ref{F1}. On such graphs the following variant of $J$ has been a popular choice, in particular for image processing applications
\begin{equation*}
\begin{aligned} 
J_{\mathrm{iso}}(u)=\sum^{N-1}_{j=1}\sum^{M-1}_{i=1}\sqrt{\left|u(v_{i+1,j})-u(v_{i,j})\right|^{2}+\left|u(v_{i,j+1})-u(v_{i,j})\right|^{2}}+\\
\sum^{M-1}_{i=1}\left|u(v_{i+1,N})-u(v_{i,N})\right|+\sum^{N-1}_{j=1}\left|u(v_{M,j+1})-u(v_{M,j})\right|,
\end{aligned}
\end{equation*}
see, for instance, \cite{AujAubBlaCha05,AujGilChaOsh06,Chambolle2}. It can be shown that $J_{\mathrm{iso}}$ is the support function of $\divergence \mathcal{B}^{\text{iso}}_{1}$, where
\begin{align*}
	\mathcal{B}^{\text{iso}}_{1} = \left\{H\in \mathbb{R}^{E}: \max_{i,j} C_{ij}(H) \leq 1 \right\},
\end{align*}
and $C_{ij}(H)$ is given by
\begin{equation*}
	C_{ij}(H)=
 	\begin{cases}
 		\sqrt{H((v_{i+1,j},v_{i,j}))^{2}+H((v_{i,j+1},v_{i,j}))^{2}},	& i \le M-1 ,\; j \le N-1,\\
 		|H((v_{i+1,N},v_{i,N}))|,	& i \le M-1,\; j=N, \\
		|H((v_{M,j+1},v_{M,j}))|,	& i=M,\; j \le N-1,\\
		0, & i=M,\; j=N.
 	\end{cases} 
\end{equation*}
Let $M,N>1$. From the construction of $\mathcal{B}^{\text{iso}}_{1}$ it follows that $\partial J_{\text{iso}} (0) = \divergence \mathcal{B}^{\text{iso}}_{1}$ is not a polytope and therefore, by Theorem \ref{thm:phimin}, it cannot be invariant $\varphi$-minimal. Consequently, the minimizer of the isotropic ROF model, which can be characterized as
	$$ \argmin_{u\in f-\alpha\partial J_{\text{iso}}(0)} \|u\|_2,$$
in general does not have property \eqref{TVRP}.
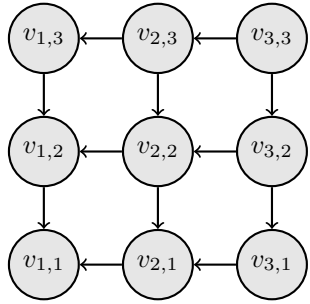
\begin{figure}[h]
\centering
 \begin{tikzpicture}[scale=0.75]
\tikzstyle{every node}=[draw, circle, thick, fill=gray!20, minimum width={width("fjhere")},]

	\node       	 (a) at (0,0)  {$v_{1,2}$};
	\node            (b) at (2,0)  {$v_{2,2}$};
	\node            (c) at (4,0)  {$v_{3,2}$};
	\node            (d) at (2,2)  {$v_{2,3}$};
	\node            (e) at (2,-2) 	{$v_{2,1}$};
	\node            (f) at (0,2)  {$v_{1,3}$};
	\node            (g) at (0,-2) 	{$v_{1,1}$};
	\node            (h) at (4,-2) 	{$v_{3,1}$};
	\node            (i) at (4,2)  {$v_{3,3}$};
	\draw[<-,thick] (a) edge (b);
	\draw[<-,thick] (b) edge (c);
	\draw[<-,thick] (b) edge (d);
	\draw[<-,thick] (e) edge (b);
	\draw[<-,thick] (a) edge (f);
	\draw[<-,thick] (g) edge (a);
	\draw[<-,thick] (f) edge (d);
	\draw[<-,thick] (g) edge (e);
	\draw[<-,thick] (d) edge (i);
	\draw[<-,thick] (c) edge (i);
	\draw[<-,thick] (h) edge (c);
	\draw[<-,thick] (e) edge (h);
\end{tikzpicture}
\caption{A $3 \times 3$ Cartesian graph.}
\label{F1}
\end{figure}
\end{remark}
\begin{remark}
In the continuous setting it is known that an analogue of Theorem \ref{T1} holds for isotropic total variation, see \cite[Thm.\ 4.46]{Scherzer1}.
\end{remark}

\subsection{Further properties of the ROF minimizer} \label{S3}
In this subsection we study further properties of the ROF minimizer $u_{\alpha}$. We first give an auxiliary result.
\begin{lemma} \label{thm:convexcombination}
 Let $0\leq\beta_{1}<\beta_{2}$. If $\partial J(u_{\beta_{1}})=\partial J(u_{\beta_{2}})$, then for every $\alpha \in (\beta_1,\beta_2)$ the ROF minimizer $u_\alpha$ is a convex combination of $u_{\beta_1}$ and $u_{\beta_2}$. That is,
 \begin{align} \label{u(alpha)}
   u_{\alpha} = \frac{\beta_{2}-\alpha}{\beta_{2}-\beta_{1}}u_{\beta_{1}} + \frac{\alpha-\beta_{1}}{\beta_{2}-\beta_{1}}u_{\beta_{2}},\quad \beta_{1} < \alpha < \beta_{2}.
 \end{align}
\end{lemma}
\begin{proof}
Denote the convex combination in \eqref{u(alpha)} by $c(\alpha)$. It suffices to verify that $c(\alpha)$ satisfies the optimality conditions \eqref{OC1}, that is, $f - c(\alpha) \in \alpha \partial J(c(\alpha))$. First, note that by item 2 in Remark \ref{rem:subdifequal1} we have $\partial J(c(\alpha)) = \partial J(u_{\beta_1}).$ Next, let $u^*_{\beta_{i}} = f - u_{\beta_i}$, $i=1,2$.

If $\beta_1>0$, we compute
\begin{align*}
	\frac{f-c(\alpha)}{\alpha}
		&= \frac{1}{\alpha} \left[ \frac{\beta_{2}-\alpha}{\beta_{2}-\beta_{1}}u^*_{\beta_{1}}+\frac{\alpha-\beta_{1}}{\beta_{2}-\beta_{1}}u^*_{\beta_{2}} \right] \\
		&= \frac{\beta_1}{\alpha} \frac{\beta_{2}-\alpha}{\beta_{2}-\beta_{1}} \frac{u^*_{\beta_{1}}}{\beta_1} +
		   \frac{\beta_2}{\alpha} \frac{\alpha-\beta_{1}}{\beta_{2}-\beta_{1}} \frac{u^*_{\beta_{2}}}{\beta_2}.
\end{align*}
It is straightforward to check that the last expression is a convex combination of $u^*_{\beta_{1}}/\beta_1$ and $u^*_{\beta_{2}}/\beta_2$. By optimality of $u_{\beta_i}$ and the assumption that $\partial J(u_{\beta_{1}})=\partial J(u_{\beta_{2}})$, both $u^*_{\beta_{i}}/\beta_i$ lie in the same convex set $\partial J(u_{\beta_1}).$ Therefore $(f-c(\alpha))/\alpha$ is in this set, too. We conclude that $c(\alpha)$ must be the ROF minimizer $u_\alpha$.

If $\beta_1=0$, then $u^*_{\beta_1} = 0$ and $(f-c(\alpha))/\alpha = u^*_{\beta_2}/\beta_2 \in \partial J(c(\alpha)).$
\end{proof}
We can now show the following properties of the ROF minimizer.
\begin{proposition} \label{P7}
\leavevmode
\begin{enumerate}
 \item Problem \eqref{ROFG} is mean-preserving, that is
 $$ \sum_{v\in V} u_{\alpha}(v) = \sum_{v\in V} f(v) \quad \text{for all }\alpha\ge 0.$$
 \item The function $\alpha \mapsto	\| u_{\alpha} \|_{2}$ is nonincreasing on $\left[0,\infty\right)$.
\item The solution $u_{\alpha}$ is a continuous piecewise affine function with respect to $\alpha$. Its piecewise constant derivative $d  u_{\alpha}/d \alpha$ exists everywhere except for a finite number of values of $0<\alpha_{1}<...<\alpha_{N}<\infty$. In particular,
 \begin{align}\label{eq:stationaryrof}
  u_{\alpha}(v)=\frac{1}{\left|V\right|}\sum_{w\in V}f(w),\quad \text{for all }\alpha\geq\alpha_{N}\text{ and }v\in V.
 \end{align}
\end{enumerate}
\end{proposition}
\begin{proof}
\leavevmode
\begin{enumerate}
\item According to Proposition \ref{thm:rofconstrained} we have $u_{\alpha}=f-\divergence H$ for an $H\in\mathbb{R}^E$. Summing this equation over all $v\in V$ and using the fact that $\sum_{v\in V} \divergence H(v)$ vanishes for every $H\in \mathbb{R}^E$ gives $\sum_{v\in V} u_{\alpha}(v) = \sum_{v\in V} f(v)$ for all $\alpha\geq 0$.
\item From the dual formulation of the ROF model, we know that $u_{\alpha}$ is the $\ell^{2}$-minimizer in the set $f-\divergence\mathcal{B}_{\alpha}$. Since $f-\divergence\mathcal{B}_{\beta_{1}}\subset f-\divergence\mathcal{B}_{\beta_{2}}$, $\beta_{1}\leq\beta_{2}$, it then follows that $\alpha\mapsto\| u_{\alpha} \|_{2}$ is nonincreasing.
\item We first prove that the map $\alpha \mapsto u_{\alpha}$ is continuous.
Consider a convergent sequence of regularization parameters $\alpha_n \to \alpha$. According to the optimality condition \eqref{OC1} the corresponding minimizers $u_n \coloneqq u_{\alpha_n}$ and $u \coloneqq u_{\alpha}$ can be expressed as
\begin{align*}
	u_{n}	&=	f - \alpha_n u_n^*, \\
	u	&=	f - \alpha u^*,
\end{align*}
for certain $u^*_n \in \partial J(u_{n})$ and $u^* \in \partial J(u)$. We compute
\begin{align*}
	\| u_{n} - u \|_2^2
		&= \langle u_{n} - u, u_{n} - u \rangle_{\mathbb{R}^V} \\
		&= \langle u_{n} - u, \alpha u^* - \alpha_n u^*_n \rangle_{\mathbb{R}^V} \\
		&= \alpha \langle u_{n}, u^*\rangle_{\mathbb{R}^V} -  \alpha_n \langle u_{n} , u^*_n \rangle_{\mathbb{R}^V} - \alpha \langle u, u^* \rangle_{\mathbb{R}^V} + \alpha_n \langle u, u^*_n \rangle_{\mathbb{R}^V}.
	\intertext{Using the fact that $\langle u,u^* \rangle = J(u)$ and $\langle u_n,u_n^* \rangle = J(u_n)$ while $\langle u_n,u^* \rangle \le J(u_n)$ and $\langle u,u_n^* \rangle \le J(u)$ according to Lemma \ref{L1}, we obtain}
	\| u_{n} - u \|_2^2
		&\le \alpha J(u_{n}) - \alpha_n J(u_{n}) - \alpha J(u) + \alpha_n J(u) \\
		&\le |\alpha_n - \alpha| |J(u) + J(f)|,
\end{align*}
and therefore $u_n \to u$.

The piecewise affine structure of $u_{\alpha}$ has been shown in \cite[Thm.~4.6]{Burger1}. However, since our proof relies on different arguments, we choose to include it.

From Lemma \ref{thm:convexcombination} as well as Remark \ref{rem:subdifequal1}, items 2 and 3, we can derive two important facts. These two facts, combined with continuity of the map $\alpha \mapsto u_\alpha$, show that it must be piecewise affine on $[0,\infty)$. First, the subdifferential $\partial J(u_\alpha)$ can only change a finite number of times. Second, in intervals where it does not change, the minimizer $u_\alpha$ is an affine function of $\alpha$.
\par Finally, consider $u_{\alpha}$ for $\alpha\geq\alpha_{N}$, where $\alpha_{N}$ is the last time $\partial J (u_\alpha)$ changes. Let $\bar{f}$ denote the averaged initial image $f$, i.e.
 \begin{align}\label{eq:faverage}
  \bar{f}(v)=\frac{1}{\left|V\right|}\sum_{w\in V}f(w), \quad \text{for all }v\in V.
 \end{align}
For $\alpha\geq C$, where $C>0$ is chosen large enough, it follows that $\bar{f}\in f-\divergence\mathcal{B}_{\alpha}$. Clearly, $\bar{f}$ is the $\ell^{2}$-minimizer in $f-\divergence\mathcal{B}_{\alpha}$. Combined with the piecewise affine structure of $u_{\alpha}$, we conclude that $u_{\alpha}=\bar{f}$ for $\alpha\geq\alpha_{N}$.
\end{enumerate}
\end{proof}
\begin{remark}
Recall that in Section \ref{S1} we have assumed the graph to be connected. If this assumption is dropped, then \eqref{eq:stationaryrof} does not hold in general, since $\bar f$ might not be a minimizer for any $\alpha.$ If the graph is disconnected, however, the ROF problem decouples into mutually independent subproblems, one for each connected component of the graph. Statement \eqref{eq:stationaryrof} then applies to each subproblem. An analogous remark can be made about property \eqref{eq:stationaryflow} of the TV flow.
\end{remark}

\section{The TV flow on the graph}\label{sec:tvflow}
In this section we consider the gradient flow associated to $J$. That is, given an initial datum $f:V\to \mathbb{R}$ we want to find a function $u:[0,\infty) \to \mathbb{R}^V$ that solves the Cauchy problem
\begin{equation} \label{TVFG}
	\begin{aligned}
		u'(t)	&\in 	-\partial J(u(t)) \quad \text{for a.e. } t > 0,\\
		u(0)	&=		f.
	\end{aligned}
\end{equation}
The statements in the next theorem follow from general results on nonlinear evolution equations and semigroup theory. See \cite[Chap.\ 4]{Barbu1} for a detailed treatment and \cite[Sec.\ 2.1]{Santambrogio1} for a brief introduction to the finite-dimensional setting.
\begin{theorem}\label{thm:flowwellposed}
Solutions to problem \eqref{TVFG} have the following properties.
\begin{enumerate}
	\item For every $f\in \mathbb{R}^V$ there is a unique solution and this solution depends continuously on $f$. In particular, if $u_1$ and $u_2$ are two solutions corresponding to initial conditions $f_1$ and $f_2$, respectively, then
		$$ \| u_1(t) - u_2(t)\|_{2} \le \| u_1(s) - u_2(s) \|_{2} \quad \text{for all } t \ge s \ge 0.$$
	\item The solution $u$ lies in $C([0,\infty),\mathbb{R}^V) \cap W^{1,\infty}([0,\infty),\mathbb{R}^V)$ and satisfies
		$$ \|u'(t)\|_{2} \le \| \partial^\circ J(f) \|_{2} \quad \text{for a.e. } t \ge 0.$$
	\item The solution is right differentiable everywhere. Its right derivative is right continuous, it satisfies
		\begin{equation}\label{eq:rightderivative}
			\frac{d^+}{dt} u(t)	= -\partial^\circ J(u(t)), \quad \text{for all } t \ge 0,
		\end{equation}
		and the map
		$$ t \mapsto \Big\| \frac{d^+}{dt} u(t) \Big\|_{2} $$
		is nonincreasing.
	\item Define $S_t(f) = u(t)$. Then, for every $f\in \mathbb{R}^V$, we have
		$$ S_t (S_s (f)) = S_{t+s}(f) \quad \text{for all } t,s \ge 0.$$
	\item The function $u(t) \in  \mathbb{R}^V$ converges to a minimizer of $J$ as $t\to \infty$. 
\end{enumerate}
\end{theorem}
Equation \eqref{eq:rightderivative} is a strengthening of the inclusion in \eqref{TVFG}. It implies, for instance, that whenever $u'$ exists, it equals $-\partial^\circ J (u)$. Note that Theorem \ref{thm:flowwellposed} actually holds true for any convex real-valued functional, which admits a minimizer on $\mathbb{R}^V$, in place of $J$. For $J$ being the total variation, however, we have in addition the following analogue of Proposition \ref{P7}.
\begin{proposition}\label{thm:flowproperties}
\leavevmode
\begin{enumerate}
	\item Problem \eqref{TVFG} is mean-preserving, that is,
		$$ \sum_{v\in V} u(t)(v) = \sum_{v\in V} f(v) \quad \text{for all } t \ge 0.$$
	\item The function $t \mapsto	\| u(t) \|_{2}$	is nonincreasing on $[0,\infty)$.
	\item The solution $u$ is piecewise affine with respect to $t$. More specifically, the derivative $u'(t)$ does not exist for only a finite number of times $0<t_1 < \cdots < t_M$ and it is constant in between. It follows that a stationary solution is reached in finite time:
	\begin{equation}\label{eq:stationaryflow}
		u(t)(v) = \frac{1}{|V|}\sum_{w\in V} f(w)  \quad \text{for all } t \ge t_M \text{ and } v\in V.	
	\end{equation}
\end{enumerate}
\end{proposition}
\begin{proof}
\leavevmode
\begin{enumerate}
	\item Since the subdifferential of $J$ consists entirely of divergences of edge functions, for a.e.\ $t\ge 0$ there is an $H(t) \in \mathbb{R}^E$ such that
		\begin{equation*}
			u'(t) = -\divergence H(t).
		\end{equation*}
		Summing this equation over all $v\in V$ and using the fact that $\sum_{v\in V} \divergence H(v)$ vanishes for every $H\in \mathbb{R}^E$ gives
		$$ \frac{d}{dt} \sum_{v\in V} u(t)(v) = 0 \quad \text{for a.e. } t\ge 0.$$
		Since $u\in W^{1,\infty}([0,\infty),\mathbb{R}^V)$, the assertion follows.		
	\item From $-u'(t) \in \partial J(u(t))$ and the characterization of the subdifferential in Lemma \ref{L1}, it follows that
		$\langle u(t), -u'(t) \rangle_{\mathbb{R}^V} = J(u(t)).$ Therefore
		$$-J(u(t)) = \langle u(t), u'(t) \rangle_{\mathbb{R}^V} = \frac{1}{2} \frac{d}{dt} \|u(t)\|^2_{2}$$
		for a.e.~$t>0$, which shows that $t \mapsto \| u(t) \|_{2}$ is nonincreasing.	
	\item As for the ROF minimizer the piecewise affine behaviour has been shown in \cite[Thm.~4.6]{Burger1}. Our proof uses different arguments. According to item 3 in Remark \ref{rem:subdifequal1} the number of different values the right derivative of $u$ can take is finite. Since $d^+u/dt$ is also right continuous, there must be an $\epsilon>0$ for every $t_0 \ge 0$ such that
		$$ \frac{d^+}{dt} u(t) = -\partial^\circ J(u(t_0)) \quad \text{for all } t \in [t_0,t_0+\epsilon)$$
with $d^+u/dt = u'$ on $(t_0,t_0+\epsilon)$. This proves that $t \mapsto u(t)$ is piecewise affine on $[0,\infty).$

That $d^+u/dt$ only changes a finite number of times follows from the fact that, if it changes, then its norm becomes strictly smaller. To see this let $\hat t>0$ and assume that $d^+u(t)/dt \equiv c$ is constant on $(\hat t -\epsilon,\hat t)$ for some $\epsilon>0$ and that $d^+u(\hat t)/dt \neq c$. We now have
\begin{align*}
	J(u(\hat t)) =	\lim_{t\to \hat{t}^-} J(u(t)) =	\lim_{t\to \hat{t}^-} \langle u(t), -c \rangle  = \langle u(\hat t), -c \rangle,
\end{align*}
and therefore $-c \in \partial J(u(\hat t))$. However, since $-c = \partial^\circ J(u(t))$ for $t\in (\hat t - \epsilon, \hat{t})$ and the minimal section is the unique element of minimal norm in the subdifferential, we must have $\|d^+u(t)/dt\|_2 > \|d^+u(\hat t)/dt\|_2$. This combined with the fact that $d^+u/dt$ can take only a finite number of values, implies that it can change only a finite number of times.

Thus $t\mapsto u(t)$ is a continuous piecewise affine function with a finite number of slope changes. Since, by item 5 in Theorem \ref{thm:flowwellposed}, $u(t)$ is convergent, it must reach its limit in finite time. Due to mean preservation, this limit has to be the averaged initial datum.
\end{enumerate}
\end{proof}

\section{Comparison of TV regularization and TV flow} \label{sec:comparison}
In this section we first provide and analyze various conditions for the equivalence of TV regularization and TV flow on graphs. We then show that they are non-equivalent methods by constructing a counterexample.

\subsection{Conditions for equivalence of TV regularization and TV flow}
Proposition \ref{thm:norm} below relates the norms of the solutions of the TV regularization and the TV flow to each other. Recall that $\bar{f}$ denotes the averaged datum $f$, see \eqref{eq:faverage}.
\begin{proposition} \label{thm:norm}
For every $\alpha>0$ let $u_\alpha$ and $u(\alpha)$ be the ROF and TV flow solutions, respectively, both corresponding to the same datum $f\in \mathbb{R}^V$. They satisfy
$$ \|\bar f \|_{2} \le \|u_\alpha\|_{2} \le \| u(\alpha) \|_{2} \le \| f \|_{2}, \quad \text{for all } \alpha>0.$$
It follows that in general $u_\alpha$ reaches $\bar f$ before $u(t)$, that is, $\alpha_N \le t_M$, see Propositions \ref{P7} and \ref{thm:flowproperties}.
\end{proposition}
\begin{proof}
Both $\|u_\alpha\|_{2}$ and $\|u(\alpha)\|_{2}$ are nonincreasing functions of $\alpha$, recall property 2 in Propositions \ref{P7} and \ref{thm:flowproperties}, and therefore bounded from above by $\| f \|_{2}$. On the other hand, due to mean preservation, recall property 1 in Propositions \ref{P7} and \ref{thm:flowproperties}, they are bounded from below by $\|\bar f \|_{2}$. It remains to show that $\|u_\alpha\|_{2} \le \| u(\alpha) \|_{2}$. To see this, observe that both $u_\alpha$ and $u(\alpha)$ lie in $f-\divergence\mathcal{B}_{\alpha}$ with $u_\alpha$ being the element of minimal norm in this set according to \eqref{206}.
\end{proof}
The next proposition collects several conditions for equality of ROF and TV flow solutions. The second condition is an adaptation of \cite[Thm.\ 10]{Lasica1} to the graph setting.
\begin{proposition}\label{thm:eqcond}
Let $u_\alpha$ and $u(t)$ be the ROF and TV flow solutions for a given common datum $f\in \mathbb{R}^V$.
\begin{enumerate}
	\item Let $\alpha> 0.$ We have $u_\alpha = u(\alpha)$ if and only if
		 \begin{align} \label{eq:necsuffcond}
 			 -\frac{1}{\alpha}\int^{\alpha}_{0}u^{\prime}(t)dt\in\partial J(u(\alpha)).
 		\end{align}
 	\item Let $\alpha> 0.$ If
 	\begin{align}\label{eq:suffcond}
		-\langle u'(t), u(\alpha) \rangle_{\mathbb{R}^V} = J(u(\alpha)) \quad \text{for a.e. } t \in (0,\alpha)
	\end{align}
		then $u(\alpha) = u_{\alpha}$. Moreover, condition \eqref{eq:suffcond} is always satisfied for $\alpha=t_1$, where $t_1$ is the first time $u'(t)$ does not exist.
	\item Define $T_\alpha(f) = u_\alpha$. We have
		$$u_\alpha = u(\alpha) \quad \text{for all} \quad \alpha \ge 0$$
		if and only if
 	\begin{align} \label{eq:semigrouprof}
  		T_{t}(T_{s}(f))=T_{t+s}(f) \quad \text{for all} \quad  t,s\geq 0.
 	\end{align}
\end{enumerate}
\end{proposition}
\begin{proof}
\leavevmode
\begin{enumerate}
	\item We can express $u(\alpha)=f+\int^{\alpha}_{0}u^{\prime}(t)dt$. Recalling the optimality conditions \eqref{OC1} for the ROF minimizer $u_{\alpha}$, it follows that $u(\alpha)=u_{\alpha}$ if and only if $-\frac{1}{\alpha}\int^{\alpha}_{0}u^{\prime}(t)dt\in\partial J(u(\alpha))$.
	\item The proof is analogous to the one of \cite[Thm.\ 10]{Lasica1}. We include it for the sake of completeness.

Integrating \eqref{eq:suffcond} from $t=0$ to $t=\alpha$ gives
	\begin{equation}\label{eq:suffcond1}
		\langle f - u(\alpha), u(\alpha) \rangle_{\mathbb{R}^V} = \alpha J(u(\alpha)).
	\end{equation}
On the other hand, since $-u'(t)$ lies in $\partial J (0)$ for almost every $t$, so does its average $-\frac{1}{\alpha}\int_0^\alpha u'(t)\,dt$. Therefore
	\begin{equation}\label{eq:suffcond2}
		f - u(\alpha) = -\int_0^\alpha u'(t)\,dt \in \alpha \partial J (0).
	\end{equation}
Combining \eqref{eq:suffcond1} and \eqref{eq:suffcond2} shows that $f - u(\alpha) \in \alpha \partial J(u(\alpha))$, recall Lemma \ref{L1}. But this is just the optimality condition \eqref{OC1} for the ROF model, hence $u(\alpha) = u_\alpha.$

Next, recall that the flow solution satisfies
	$$ -u'(t) = \partial^\circ J (f) \in \partial J (u(t)), \quad t \in [0,t_1).$$
This implies by Lemma \ref{L1} that
	$$ \langle \partial^\circ J (f),u(t) \rangle = J(u(t)), \quad t \in [0,t_1),$$
and since $u$ is continuous in $t$
	$$ \langle \partial^\circ J (f),u(t_1) \rangle = J(u(t_1)).$$
Therefore condition \eqref{eq:suffcond} is satisfied for every $\alpha\in\left[0,t_1\right]$.
	\item Let $u_{\alpha}=u(\alpha)$ for all $\alpha\geq 0$. It then follows from property 4 in Theorem \ref{thm:flowwellposed} that $T_{t}(T_{s}(f))=T_{t+s}(f)$ for all $t,s\geq 0$.
\par Start now with the assumption $T_{t}(T_{s}(f))=T_{t+s}(f)$ for all $t,s\geq 0$. As the TV flow has an analogous property and the solutions to TV regularization and TV flow always coincide for the interval $\left[0,t_{1}\right]$ according to item 2 it is then immediate that they coincide for all $\alpha\geq 0$.
\end{enumerate}
\end{proof}
\begin{remark}\label{rem:eqcond}
\leavevmode
\begin{enumerate}
	\item Proposition \ref{thm:eqcond}, item 1, gives that $u(\alpha)=u_{\alpha}$ if and only if the average time derivative $\frac{1}{\alpha}\int^{\alpha}_{0}u^{\prime}(t)dt$ is in $-\partial J(u(\alpha))$. Compare with the pointwise inclusion $u^{\prime}(t)\in-\partial J(u(t))$ which holds for a.e. $t>0$. Note further that condition \eqref{eq:necsuffcond} is strictly weaker than \eqref{eq:suffcond}.
	\item Condition \eqref{eq:suffcond} holds true, given any $\alpha>0$, for graphs of the type displayed in Figure \ref{Figure1D} corresponding to one-dimensional space-discrete signals. This follows directly from the inclusion
 \begin{align} \label{inc}
 \partial J(u(s))\subset\partial J(u(t)),\quad s \leq t,
 \end{align}
which applies in this setting. The derivation of \eqref{inc} can be done with the following arguments. Consider a pair of adjacent vertices $v_{i}$ and $v_{i+1}$. In \cite[Prop.\ 4.1]{Steidl1}, it is shown that if $u(s)(v_{i})=u(s)(v_{i+1})$ then $u(t)(v_{i})=u(t)(v_{i+1})$ for any $t\geq s$. Taking into account the continuity of $t \mapsto u(t)$ and the characterization of the subdifferential given by item 3 in Lemma \ref{L1}, \eqref{inc} then follows.
\item Another family of instances where $u_\alpha = u(\alpha)$, for all $\alpha \ge 0$, arises from the eigenvalue problem for the TV subdifferential. This problem seems to have originally been studied in the continuous setting, where it was realized to give rise to explicit solutions of both the TV flow and the ROF model. See, for instance, \cite{AndCasDiaMaz02,BelCasNov02}. In the discrete setting the situation is similar. Following \cite{Burger1,Gil14} we call $f\in \mathbb{R}^V$ an \emph{eigenfunction} of $J$, if it satisfies $\lambda f \in \partial J (f)$ for some $\lambda \ge 0$. If the datum of the ROF model has this property, then the optimality condition \eqref{OC1} directly implies that
\begin{equation*}
	u_\alpha =
	\begin{cases}
		(1-\alpha \lambda) f,	& \alpha\lambda < 1, \\
		0,						& \alpha\lambda \ge 1.
	\end{cases}
\end{equation*}
See also \cite[Thm.\ 5]{BenBur13}. In other words $T_\alpha(f)$ is a nonnegative multiple of $f$, hence again an eigenfunction. A brief calculation now shows that \eqref{eq:semigrouprof} is satisfied.
\end{enumerate}
\end{remark}
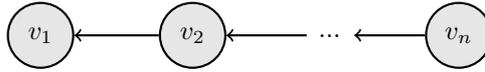
\begin{figure}
\centering
 \begin{tikzpicture}[scale=1]
 \node		(d) at (4,0) {};
 \node		    at (3.8,0) {...};
\tikzstyle{every node}=[draw, circle, thick, fill=gray!20, minimum width={width("fjhere")},]

	\node       	 (a) at (0,0)  {$v_{1}$};
	\node            (b) at (2,0)  {$v_{2}$};
	\node            (c) at (5.5,0)  {$v_{n}$};
	\draw[<-,thick] (a) edge (b);
	\draw[<-,thick] (b) edge (3.5,0);
 	\draw[<-,thick] (d) edge (c);
\end{tikzpicture}
\caption{Graph corresponding to a one-dimensional space-discrete signal with $n$ pixels.}
\label{Figure1D}
\end{figure}

\subsection{Negative results}\label{sec:example}
All results in this section are derived from the counterexample given by the graph and datum displayed in Figure \ref{Figure1}. While the corresponding solutions $u_\alpha$ and $u(t)$ are illustrated in Figures \ref{Figure2} and \ref{Figure3}, the underlying computations can be found in the appendix. Our main considerations in constructing this counterexample are explained below.

Proposition \ref{thm:rofconstrained} together with the fact that $\mathcal{B}_{\alpha} = -\mathcal{B}_{\alpha}$ implies that the ROF minimizer $u_{\alpha}$ can be written as
\begin{align} \label{reprof}
	u_{\alpha}=f+\divergence F_{\alpha},
\end{align}
for an $F_\alpha \in \mathcal{B}_{\alpha}$.
Regarding the TV flow, note that Lemma \ref{L1} and Proposition \ref{thm:flowproperties} guarantee the existence of a piecewise constant function $t \mapsto H(t) \in \mathcal{B}_{1,u(t)}$ with finitely many discontinuities satisfying
$$ u'(t) = - \divergence H(t)$$
for all but a finite number of times. Integrating and setting $F(t) = -\int_0^t H(s)ds$ we obtain the following representation
\begin{align}\label{repflow}
	u(t) = f + \divergence F(t).
\end{align}
Two properties concerning these representations are worth mentioning. First, the edge functions $F_\alpha$ and $F(t)$ are not uniquely determined in general. Second, $F(t)$ satisfies
$$\Big\|\frac{d^+}{dt} F(t) \Big\|_\infty \le 1$$
for all $t$, while the derivative of $F_\alpha$ in general is not bounded by one. The counterexample displayed in Figure \ref{Figure1} was constructed in such a way that $F_\alpha$ is uniquely determined and satisfies $\|dF_\alpha/d\alpha\|_\infty > 1$ for certain values of $\alpha.$ In fact, on the edge $e=(v_{32},v_{22})$ we have $dF_\alpha(e)/d\alpha = -3/2$ for $2/5 < \alpha < 2$, see Figure \ref{Figure2}.

\subsubsection{Nonequivalence of TV flow and TV regularization}
In spite of the similar qualitative properties of TV flow and TV regularization, recall Propositions \ref{P7} and \ref{thm:flowproperties}, the solutions $u(\alpha)$ and $u_{\alpha}$ do not coincide in general.
\begin{theorem} \label{prop:noneq}
\leavevmode
There exist graphs $G=(V,E)$ and data $f\in\mathbb{R}^{V}$ for which the TV regularization problem and the TV flow problem are non\-equivalent, i.e.
 \begin{align*}
  u_{\alpha}\neq u(\alpha),\text{ for some }\alpha>0.
 \end{align*}
\end{theorem}
\begin{proof}
 Consider the graph and the datum $f$ given in Figure \ref{Figure1}. For this example, the evolutions of $u_{\alpha}$ and $u(\alpha)$ on the interval $\left[0,4\right]$ are displayed in Figure \ref{Figure2} and Figure \ref{Figure3}, respectively. Note that $u_{\alpha}\neq u(\alpha)$ for $\alpha\in\left(2/5,4\right]$.
\end{proof}
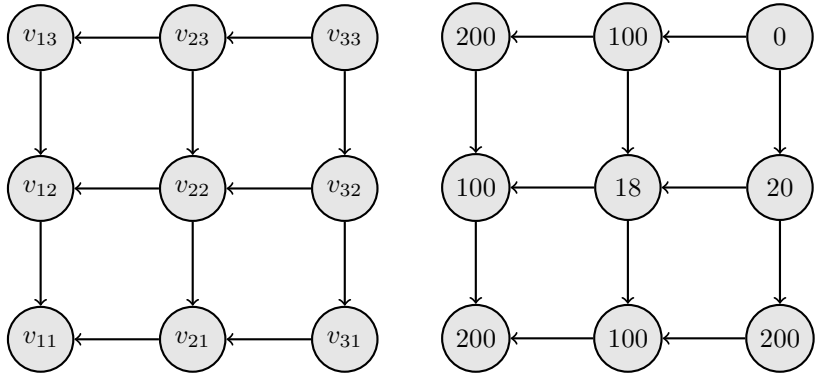
\begin{figure}
\centering
 \begin{tikzpicture}[scale=1]
\tikzstyle{every node}=[draw, circle, thick, fill=gray!20, minimum width={width("fjhere")},]

	\node       	 (a) at (0,0)  {$v_{12}$};
	\node            (b) at (2,0)  {$v_{22}$};
	\node            (c) at (4,0)  {$v_{32}$};
	\node            (d) at (2,2)  {$v_{23}$};
	\node            (e) at (2,-2) 	{$v_{21}$};
	\node            (f) at (0,2)  {$v_{13}$};
	\node            (g) at (0,-2) 	{$v_{11}$};
	\node            (h) at (4,-2) 	{$v_{31}$};
	\node            (i) at (4,2)  {$v_{33}$};
	\draw[<-,thick] (a) edge (b);
	\draw[<-,thick] (b) edge (c);
	\draw[<-,thick] (b) edge (d);
	\draw[<-,thick] (e) edge (b);
	\draw[<-,thick] (a) edge (f);
	\draw[<-,thick] (g) edge (a);
	\draw[<-,thick] (f) edge (d);
	\draw[<-,thick] (g) edge (e);
	\draw[<-,thick] (d) edge (i);
	\draw[<-,thick] (c) edge (i);
	\draw[<-,thick] (h) edge (c);
	\draw[<-,thick] (e) edge (h);
\end{tikzpicture}
\qquad
 \begin{tikzpicture}[scale=1]
\tikzstyle{every node}=[draw, circle, thick, fill=gray!20, minimum width={width("fjhere")},]

	\node       	 (a) at (0,0)  {$100$};
	\node            (b) at (2,0)  {$18$};
	\node            (c) at (4,0)  {$20$};
	\node            (d) at (2,2)  {$100$};
	\node            (e) at (2,-2)  {$100$};
	\node            (f) at (0,2)  {$200$};
	\node            (g) at (0,-2)  {$200$};
	\node            (h) at (4,-2)  {$200$};
	\node            (i) at (4,2)  {$0$};
	\draw[<-,thick] (a) edge (b);
	\draw[<-,thick] (b) edge (c);
	\draw[<-,thick] (b) edge (d);
	\draw[<-,thick] (e) edge (b);
	\draw[<-,thick] (a) edge (f);
	\draw[<-,thick] (g) edge (a);
	\draw[<-,thick] (f) edge (d);
	\draw[<-,thick] (g) edge (e);
	\draw[<-,thick] (d) edge (i);
	\draw[<-,thick] (c) edge (i);
	\draw[<-,thick] (h) edge (c);
	\draw[<-,thick] (e) edge (h);
\end{tikzpicture}
\caption{Left: graph structure. Right: datum $f$.}
\label{Figure1}
\end{figure}

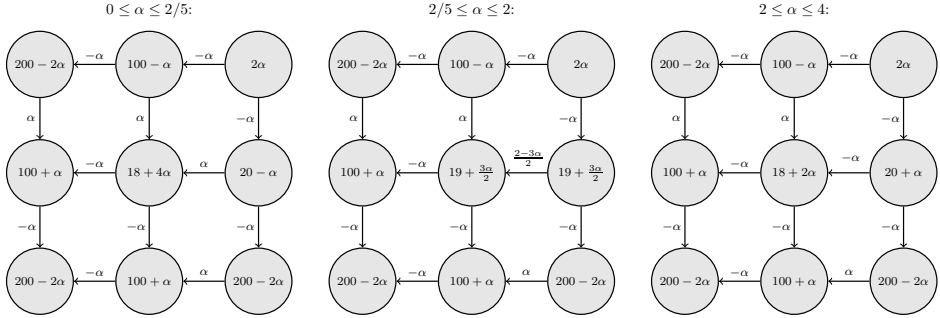
\begin{figure}
\centering
\scalebox{0.575}{
\begin{tikzpicture}
\node at (2.5,3.75) {$0\leq\alpha\leq 2/5$:};
{\footnotesize
\node       	  at (1.25,0.2)  {$-\alpha$};
\node       	  at (3.75,0.2)  {$\alpha$};
\node       	  at (1.25,2.7)  {$-\alpha$};
\node       	  at (3.75,2.7)  {$-\alpha$};
\node       	  at (1.25,-2.3)  {$-\alpha$};
\node       	  at (3.75,-2.3)  {$\alpha$};
\node       	  at (-0.3,-1.25)  {$-\alpha$};
\node       	  at (-0.2,1.25)  {$\alpha$};
\node       	  at (2.2,-1.25)  {$-\alpha$};
\node       	  at (2.3,1.25)  {$\alpha$};
\node       	  at (4.7,-1.25)  {$-\alpha$};
\node       	  at (4.7,1.25)  {$-\alpha$};
\tikzstyle{every node}=[draw, circle, thick, fill=gray!20, minimum width={width("fjhererleheh")},]
	\node       	 (a) at (0,0)  {$100+\alpha$};
	\node            (b) at (2.5,0)  {$18+4\alpha$};
	\node            (c) at (5,0)  {$20-\alpha$};
	\node            (d) at (2.5,2.5)  {$100-\alpha$};
	\node            (e) at (2.5,-2.5)  {$100+\alpha$};
	\node            (f) at (0,2.5)  {$200-2\alpha$};
	\node            (g) at (0,-2.5)  {$200-2\alpha$};
	\node            (h) at (5,-2.5)  {$200-2\alpha$};
	\node            (i) at (5,2.5)  {$2\alpha$};
	\draw[<-,thick] (a) edge (b);
	\draw[<-,thick] (b) edge (c);
	\draw[<-,thick] (b) edge (d);
	\draw[<-,thick] (e) edge (b);
	\draw[<-,thick] (a) edge (f);
	\draw[<-,thick] (g) edge (a);
	\draw[<-,thick] (f) edge (d);
	\draw[<-,thick] (g) edge (e);
	\draw[<-,thick] (d) edge (i);
	\draw[<-,thick] (c) edge (i);
	\draw[<-,thick] (h) edge (c);
	\draw[<-,thick] (e) edge (h);
	}
\end{tikzpicture}
\qquad
\begin{tikzpicture}
\node at (2.5,3.75) {$2/5\leq\alpha\leq 2$:};
{\footnotesize
\node       	  at (1.25,0.2)  {$-\alpha$};
\node       	  at (3.79,0.35)  {$\frac{2-3\alpha}{2}$};
\node       	  at (1.25,2.7)  {$-\alpha$};
\node       	  at (3.75,2.7)  {$-\alpha$};
\node       	  at (1.25,-2.3)  {$-\alpha$};
\node       	  at (3.75,-2.3)  {$\alpha$};
\node       	  at (-0.3,-1.25)  {$-\alpha$};
\node       	  at (-0.2,1.25)  {$\alpha$};
\node       	  at (2.2,-1.25)  {$-\alpha$};
\node       	  at (2.3,1.25)  {$\alpha$};
\node       	  at (4.7,-1.25)  {$-\alpha$};
\node       	  at (4.7,1.25)  {$-\alpha$};
\tikzstyle{every node}=[draw, circle, thick, fill=gray!20, minimum width={width("fjhererleheh")},]
	\node       	 (a) at (0,0)  {$100+\alpha$};
	\node            (b) at (2.5,0)  {$19+\frac{3\alpha}{2}$};
	\node            (c) at (5,0)  {$19+\frac{3\alpha}{2}$};
	\node            (d) at (2.5,2.5)  {$100-\alpha$};
	\node            (e) at (2.5,-2.5)  {$100+\alpha$};
	\node            (f) at (0,2.5)  {$200-2\alpha$};
	\node            (g) at (0,-2.5)  {$200-2\alpha$};
	\node            (h) at (5,-2.5)  {$200-2\alpha$};
	\node            (i) at (5,2.5)  {$2\alpha$};
	\draw[<-,thick] (a) edge (b);
	\draw[<-,thick] (b) edge (c);
	\draw[<-,thick] (b) edge (d);
	\draw[<-,thick] (e) edge (b);
	\draw[<-,thick] (a) edge (f);
	\draw[<-,thick] (g) edge (a);
	\draw[<-,thick] (f) edge (d);
	\draw[<-,thick] (g) edge (e);
	\draw[<-,thick] (d) edge (i);
	\draw[<-,thick] (c) edge (i);
	\draw[<-,thick] (h) edge (c);
	\draw[<-,thick] (e) edge (h);
	}
\end{tikzpicture}
\qquad
\begin{tikzpicture}
\node at (2.5,3.75) {$2\leq\alpha\leq 4$:};
{\footnotesize
\node       	  at (1.25,0.2)  {$-\alpha$};
\node       	  at (3.79,0.35)  {$-\alpha$};
\node       	  at (1.25,2.7)  {$-\alpha$};
\node       	  at (3.75,2.7)  {$-\alpha$};
\node       	  at (1.25,-2.3)  {$-\alpha$};
\node       	  at (3.75,-2.3)  {$\alpha$};
\node       	  at (-0.3,-1.25)  {$-\alpha$};
\node       	  at (-0.2,1.25)  {$\alpha$};
\node       	  at (2.2,-1.25)  {$-\alpha$};
\node       	  at (2.3,1.25)  {$\alpha$};
\node       	  at (4.7,-1.25)  {$-\alpha$};
\node       	  at (4.7,1.25)  {$-\alpha$};
\tikzstyle{every node}=[draw, circle, thick, fill=gray!20, minimum width={width("fjhererleheh")},]
	\node       	 (a) at (0,0)  {$100+\alpha$};
	\node            (b) at (2.5,0)  {$18+2\alpha$};
	\node            (c) at (5,0)  {$20+\alpha$};
	\node            (d) at (2.5,2.5)  {$100-\alpha$};
	\node            (e) at (2.5,-2.5)  {$100+\alpha$};
	\node            (f) at (0,2.5)  {$200-2\alpha$};
	\node            (g) at (0,-2.5)  {$200-2\alpha$};
	\node            (h) at (5,-2.5)  {$200-2\alpha$};
	\node            (i) at (5,2.5)  {$2\alpha$};
	\draw[<-,thick] (a) edge (b);
	\draw[<-,thick] (b) edge (c);
	\draw[<-,thick] (b) edge (d);
	\draw[<-,thick] (e) edge (b);
	\draw[<-,thick] (a) edge (f);
	\draw[<-,thick] (g) edge (a);
	\draw[<-,thick] (f) edge (d);
	\draw[<-,thick] (g) edge (e);
	\draw[<-,thick] (d) edge (i);
	\draw[<-,thick] (c) edge (i);
	\draw[<-,thick] (h) edge (c);
	\draw[<-,thick] (e) edge (h);
	}
\end{tikzpicture}
}
\caption{The evolution of the ROF minimizer $u_{\alpha}$ (on the vertices) and the function $F_{\alpha}$ (on the edges) on the interval $0\leq\alpha\leq 4$. The underlying computations can be found in the appendix.}
\label{Figure2}
\end{figure}

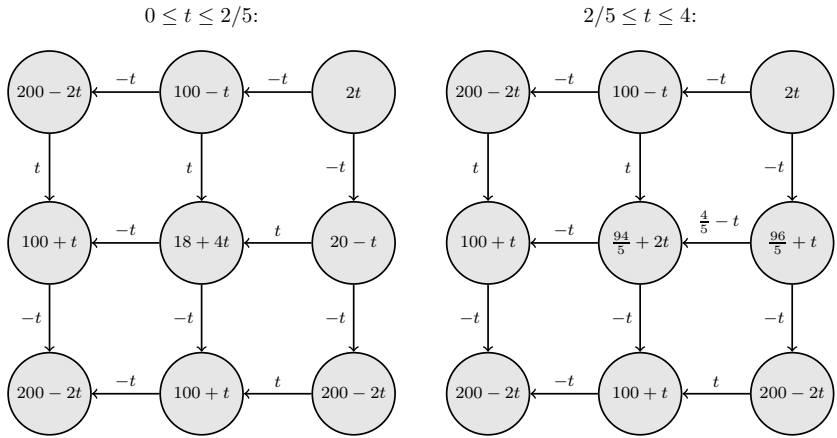
\begin{figure} 
\centering
\scalebox{0.8}{
\begin{tikzpicture}
\node at (2.5,3.75) {$0\leq t\leq 2/5$:};
{\footnotesize
\node       	  at (1.25,0.2)  {$-t$};
\node       	  at (3.75,0.2)  {$t$};
\node       	  at (1.25,2.7)  {$-t$};
\node       	  at (3.75,2.7)  {$-t$};
\node       	  at (1.25,-2.3)  {$-t$};
\node       	  at (3.75,-2.3)  {$t$};
\node       	  at (-0.3,-1.25)  {$-t$};
\node       	  at (-0.2,1.25)  {$t$};
\node       	  at (2.2,-1.25)  {$-t$};
\node       	  at (2.3,1.25)  {$t$};
\node       	  at (4.7,-1.25)  {$-t$};
\node       	  at (4.7,1.25)  {$-t$};
\tikzstyle{every node}=[draw, circle, thick, fill=gray!20, minimum width={width("fjhererlehe")},]
	\node       	 (a) at (0,0)  {$100+t$};
	\node            (b) at (2.5,0)  {$18+4t$};
	\node            (c) at (5,0)  {$20-t$};
	\node            (d) at (2.5,2.5)  {$100-t$};
	\node            (e) at (2.5,-2.5)  {$100+t$};
	\node            (f) at (0,2.5)  {$200-2t$};
	\node            (g) at (0,-2.5)  {$200-2t$};
	\node            (h) at (5,-2.5)  {$200-2t$};
	\node            (i) at (5,2.5)  {$2t$};
	\draw[<-,thick] (a) edge (b);
	\draw[<-,thick] (b) edge (c);
	\draw[<-,thick] (b) edge (d);
	\draw[<-,thick] (e) edge (b);
	\draw[<-,thick] (a) edge (f);
	\draw[<-,thick] (g) edge (a);
	\draw[<-,thick] (f) edge (d);
	\draw[<-,thick] (g) edge (e);
	\draw[<-,thick] (d) edge (i);
	\draw[<-,thick] (c) edge (i);
	\draw[<-,thick] (h) edge (c);
	\draw[<-,thick] (e) edge (h);
	}
\end{tikzpicture}
\qquad
\begin{tikzpicture}
\node at (2.5,3.75) {$2/5\leq t\leq 4$:};
{\footnotesize
\node       	  at (1.25,0.2)  {$-t$};
\node       	  at (3.79,0.35)  {$\frac{4}{5}-t$};
\node       	  at (1.25,2.7)  {$-t$};
\node       	  at (3.75,2.7)  {$-t$};
\node       	  at (1.25,-2.3)  {$-t$};
\node       	  at (3.75,-2.3)  {$t$};
\node       	  at (-0.3,-1.25)  {$-t$};
\node       	  at (-0.2,1.25)  {$t$};
\node       	  at (2.2,-1.25)  {$-t$};
\node       	  at (2.3,1.25)  {$t$};
\node       	  at (4.7,-1.25)  {$-t$};
\node       	  at (4.7,1.25)  {$-t$};
\tikzstyle{every node}=[draw, circle, thick, fill=gray!20, minimum width={width("fjhererlehe")},]
	\node       	 (a) at (0,0)  {$100+t$};
	\node            (b) at (2.5,0)  {$\frac{94}{5}+2t$};
	\node            (c) at (5,0)  {$\frac{96}{5}+t$};
	\node            (d) at (2.5,2.5)  {$100-t$};
	\node            (e) at (2.5,-2.5)  {$100+t$};
	\node            (f) at (0,2.5)  {$200-2t$};
	\node            (g) at (0,-2.5)  {$200-2t$};
	\node            (h) at (5,-2.5)  {$200-2t$};
	\node            (i) at (5,2.5)  {$2t$};
	\draw[<-,thick] (a) edge (b);
	\draw[<-,thick] (b) edge (c);
	\draw[<-,thick] (b) edge (d);
	\draw[<-,thick] (e) edge (b);
	\draw[<-,thick] (a) edge (f);
	\draw[<-,thick] (g) edge (a);
	\draw[<-,thick] (f) edge (d);
	\draw[<-,thick] (g) edge (e);
	\draw[<-,thick] (d) edge (i);
	\draw[<-,thick] (c) edge (i);
	\draw[<-,thick] (h) edge (c);
	\draw[<-,thick] (e) edge (h);
	}
\end{tikzpicture}
}

\caption{The evolution of the TV flow $u(t)$ (on the vertices) and the function $F(t)$ (on the edges) on the interval $0\leq t\leq 4$. The underlying computations can be found in the appendix.}
\label{Figure3}
\end{figure}
\begin{remark}
\leavevmode
\begin{enumerate}
	\item Proposition \ref{thm:eqcond}, item 3, combined with Theorem \ref{prop:noneq} gives that the ROF model in general does not possess the semigroup property \eqref{eq:semigrouprof}. This is in contrast to the situation for the TV flow, recall property 4 in Theorem \ref{thm:flowwellposed}.
	\item Recall Theorem \ref{thm:flowwellposed}, item 3, stating that $t \mapsto \|d^+u(t)/dt\|_2$ is nonincreasing. The ROF minimizer, in contrast, does not have an analogous property. Consider Figure \ref{Figure2}, from where it can be seen that $\lVert  du_{\alpha}/d\alpha \rVert_{{2}}$ increases from the interval $(2/5,2)$ to $(2,4)$.
	\item In \cite[Thm.\ 4.7]{Burger1} the authors give a sufficient condition for equivalence of the variational method and the gradient flow associated to a proper, convex, lower semicontinuous and absolutely one-homogeneous function $J$ on $\mathbb{R}^n$. This condition, called MINSUB, requires
	$$\langle \partial^\circ J(u) , \partial^\circ J(u) - u^* \rangle = 0$$
to hold for all $u \in \mathbb{R}^n$ and $u^* \in \partial J(u).$ Theorem \ref{prop:noneq} implies that the total variation as given in Definition \ref{def:tv} does not meet MINSUB on general graphs.
\end{enumerate}
\end{remark}

\subsubsection{Nonmonotone behaviour of jump sets}
For a given graph $G=(V,E)$ and datum $f\in \mathbb{R}^V$ we define the jump sets of the ROF and TV flow solutions in the following way
\begin{align*}
	\Gamma_\alpha	&= \left\{ (v,w) \in E :  u_\alpha(v) \neq u_\alpha(w) \right\}, \quad \alpha \ge 0, \\
	\Gamma(t)		&= \left\{ (v,w) \in E :  u(t)(v) \neq u(t)(w) \right\},\quad t \ge 0.
\end{align*}
Clearly, for $\alpha$ or $t$ large enough these two sets are empty. They do not, however, necessarily evolve in a monotone way.
\begin{proposition}\label{thm:jumpsets}
There are graphs $G=(V,E)$, data $f \in \mathbb{R}^V$ and numbers $\beta_2 > \beta_1 \ge 0$, $s_2 > s_1 \ge 0$, such that
\begin{align*}
	\Gamma_{\beta_1}	&\subsetneq \Gamma_{\beta_2}, \\
	\Gamma(s_1)			&\subsetneq \Gamma(s_2).
\end{align*}
\end{proposition}
\begin{proof}
Consider the graph and datum of Figure \ref{Figure1}.

For the TV regularization, Figure \ref{Figure2} shows that
\begin{align*}
\sgn(u_{\alpha}(v_{32})-u_{\alpha}(v_{22}))=
 \left\{
\begin{array}{rl}
1, & 0\leq\alpha< 2/5,\\
0, & 2/5\leq\alpha\leq 2,\\
-1, & 2<\alpha\leq 4.
\end{array}
\right.
\end{align*}
That is, the jump between $u_{\alpha}(v_{22})$ and $u_{\alpha}(v_{32})$ disappears for $2/5\leq\alpha\leq 2$ but appears again, with reversed sign, for $2<\alpha\leq 4$. For all other edges $(v,w)$ the quantity $\sgn(u_{\alpha}(v)-u_{\alpha}(w))$ is constant on $[0,4].$ This shows that $\Gamma_{\beta_1} \subsetneq \Gamma_{\beta_2}$ for every $\beta_1 \in [2/5,2]$ and $\beta_2 \in (2,4].$

For the TV flow, see Figure \ref{Figure3}, we have
\begin{align*}
\sgn(u(t)(v_{32})-u(t)(v_{22}))=
 \left\{
\begin{array}{rl}
1, & 0\leq t< 2/5,\\
0, & t=2/5,\\
-1, & 2/5<t\leq 4.
\end{array}
\right.
\end{align*}
Here the jump between $u(t)(v_{22})$ and $u(t)(v_{32})$ disappears at $t=2/5$ and then a jump with reversed sign appears for $2/5<t\leq 4$. Again, for all other edges $(v,w)$ the quantity $\sgn(u(t)(v)-u(t)(w))$ is constant on $[0,4].$ Thus, $\Gamma (2/5) \subsetneq \Gamma(s_2)$ for every $s_2\in (2/5,4].$
\end{proof}

\begin{remark}
For one-dimensional graphs, however, the jump sets are nonincreasing, see item 2 in Remark \ref{rem:eqcond}. On the other hand, in the continuous anisotropic setting it is known that jumps can be created in the solution, see \cite[Rem.~4]{Casalles1} and \cite[Ex.~1]{Lasica1}.
\end{remark}
\begin{remark}
We stress that $\beta_1$ and $s_1$ can be equal to zero in Proposition \ref{thm:jumpsets}. To see this consider the datum $\tilde{f}$ and solutions $u_\alpha=u(\alpha)$ given in Figure \ref{Figure4}.
\begin{figure}
\centering
\scalebox{0.8}{
 \begin{tikzpicture}
\tikzstyle{every node}=[draw, circle, thick, fill=gray!20, minimum width={width("fjhererlehe")},]
{\footnotesize
	\node       	 (a) at (0,0)  {$100$};
	\node            (b) at (2.5,0)  {$20$};
	\node            (c) at (5,0)  {$20$};
	\node            (d) at (2.5,2.5)  {$100$};
	\node            (e) at (2.5,-2.5)  {$100$};
	\node            (f) at (0,2.5) {$200$};
	\node            (g) at (0,-2.5)  {$200$};
	\node            (h) at (5,-2.5)  {$200$};
	\node            (i) at (5,2.5)  {$0$};
	\draw[<-,thick] (a) edge (b);
	\draw[<-,thick] (b) edge (c);
	\draw[<-,thick] (b) edge (d);
	\draw[<-,thick] (e) edge (b);
	\draw[<-,thick] (a) edge (f);
	\draw[<-,thick] (g) edge (a);
	\draw[<-,thick] (f) edge (d);
	\draw[<-,thick] (g) edge (e);
	\draw[<-,thick] (d) edge (i);
	\draw[<-,thick] (c) edge (i);
	\draw[<-,thick] (h) edge (c);
	\draw[<-,thick] (e) edge (h);
	}
\end{tikzpicture}
\qquad
 \begin{tikzpicture}
{\footnotesize
\node       	  at (1.25,0.2)  {$-\alpha$};
\node       	  at (3.75,0.2)  {$-\alpha$};
\node       	  at (1.25,2.7)  {$-\alpha$};
\node       	  at (3.75,2.7)  {$-\alpha$};
\node       	  at (1.25,-2.3)  {$-\alpha$};
\node       	  at (3.75,-2.3)  {$\alpha$};
\node       	  at (-0.3,-1.25)  {$-\alpha$};
\node       	  at (-0.2,1.25)  {$\alpha$};
\node       	  at (2.2,-1.25)  {$-\alpha$};
\node       	  at (2.3,1.25)  {$\alpha$};
\node       	  at (4.7,-1.25)  {$-\alpha$};
\node       	  at (4.7,1.25)  {$-\alpha$};
\tikzstyle{every node}=[draw, circle, thick, fill=gray!20, minimum width={width("fjhererlehe")},]
	\node       	 (a) at (0,0)  {$100+\alpha$};
	\node            (b) at (2.5,0)  {$20+2\alpha$};
	\node            (c) at (5,0)  {$20+\alpha$};
	\node            (d) at (2.5,2.5)  {$100-\alpha$};
	\node            (e) at (2.5,-2.5)  {$100+\alpha$};
	\node            (f) at (0,2.5)  {$200-2\alpha$};
	\node            (g) at (0,-2.5)  {$200-2\alpha$};
	\node            (h) at (5,-2.5)  {$200-2\alpha$};
	\node            (i) at (5,2.5)  {$2\alpha$};
	\draw[<-,thick] (a) edge (b);
	\draw[<-,thick] (b) edge (c);
	\draw[<-,thick] (b) edge (d);
	\draw[<-,thick] (e) edge (b);
	\draw[<-,thick] (a) edge (f);
	\draw[<-,thick] (g) edge (a);
	\draw[<-,thick] (f) edge (d);
	\draw[<-,thick] (g) edge (e);
	\draw[<-,thick] (d) edge (i);
	\draw[<-,thick] (c) edge (i);
	\draw[<-,thick] (h) edge (c);
	\draw[<-,thick] (e) edge (h);
	}
\end{tikzpicture}
}
\caption{Left: datum $\tilde{f}$. Right: $u_{\alpha}=u(\alpha)$ (on the vertices) and $F_{\alpha}=F(\alpha)$ (on the edges) for 
$\alpha \in [0,4]$.}
\label{Figure4}
\end{figure}
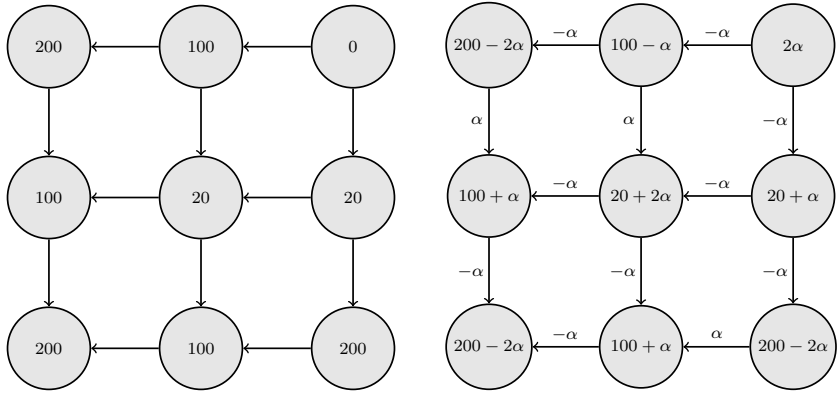
Note that $\tilde{f}$ is equal to $f$ from Figure \ref{Figure1} except for $v_{22}$ where $\tilde{f}(v_{22})=20$. The underlying calculations are analogous to the ones for $f$ and are therefore omitted.
A jump between the vertices $v_{22}$ and $v_{32}$, which is not present in the datum $\tilde{f}$, is created in $u_{\alpha}=u(\alpha)$, $0<\alpha\leq 4$. Thus the jump set of an image resulting from TV regularization or TV flow can strictly contain the jump set of the datum. 
\end{remark}

\section{Conclusion}

In this article we have studied and compared TV regularization and TV flow for functions defined on the vertices of an oriented connected graph. Our motivation was the discrete one-dimensional setting, where the two problems are known to be equivalent and their solution minimizes a large class of convex functionals in a certain neighbourhood of the data.

It turns out that in the graph setting this situation can only be recovered for $\alpha,t \in [0,t_1] \cup [t_M,\infty)$, the reason being that on the complement $(t_1,t_M)$ the ROF and flow solution are in general different. Here $t_1$ and $t_M$ are the first and last times, respectively, the time derivative of the flow solution changes.

In addition we have shown that for every $\alpha \ge 0$ the ROF minimizer $u_\alpha$ simultaneously minimizes all functionals of the form
\begin{equation}\label{eq:asf}
	\sum_{v\in V} \varphi (u(v))
\end{equation}
over the set $f - \alpha \partial J(0)$, where $\varphi:\mathbb{R}\to \mathbb{R}$ is convex but otherwise arbitrary. In doing so we have relied on the fact that $\partial J$ is invariant $\varphi$-minimal. Since invariant $\varphi$-minimal sets must be polyhedra, the subdifferential of discrete isotropic total variation cannot be such a set. Consequently, the minimizer of the isotropic ROF model in general does not have property \eqref{eq:asf}.

\subsection*{Acknowledgements}
We acknowledge support by the Austrian Science Fund (FWF) within the national research network S117 ``Geometry $+$ Simulation," subproject 4.
In addition, the work of OS is supported by project I 3661 ``Novel Error Measures and Source Conditions of Regularization Methods," jointly funded by FWF and Deutsche Forschungsgemeinschaft (DFG).
We are grateful to an anonymous referee of a previous version of this article, who pointed out the connection between invariant $\varphi$-minimal sets and submodular functions.

\appendix
\section{Appendix: TV denoising on a particular graph}\label{sec:appendix}
In this appendix we consider the graph and datum given by Figure \ref{Figure1} and compute the solutions of the TV regularization problem and the TV flow problem on the interval $\left[0,4\right]$.

\subsubsection*{TV regularization}
Recall that the ROF minimizer $u_{\alpha}$ can be represented as
\begin{align*}
 u_{\alpha}=f+\divergence F_{\alpha},
\end{align*}
where $F_{\alpha}\in\mathcal{B}_{\alpha}$, see equation \eqref{reprof}. Below, $F_{\alpha}$ is computed for $\alpha\in\left[0,4\right]$ which then enables computation of $u_{\alpha}$ on this interval.
\par We have for any $v\in V$,
\begin{align} \label{inrof}
  f(v)-\deg(v)\alpha\leq u_{\alpha}(v)\leq f(v)+\deg(v)\alpha,
\end{align}
where $\deg(v)$ denotes the degree of $v$, that is, the number of edges incident to $v$. Using \eqref{inrof} it is straightforward to show that
\begin{align*} 
\sgn(u_{\alpha}(v_{ij})-u_{\alpha}(v_{kl}))=\sgn(f(v_{ij})-f(v_{kl})) \in \{\pm 1\}
\end{align*}
for all edges $(v_{ij},v_{kl})$ except $(v_{32},v_{22})$ on the interval $0\leq\alpha\leq 4$. The optimality condition \eqref{OC1} together with the equality $\partial J(u)=\divergence\mathcal{B}_{1,u}$ (recall Lemma \ref{L1}, item 3) then gives
\begin{align*}
F_{\alpha}((v_{ij},v_{kl}))=\alpha\sgn(f(v_{ij})-f(v_{kl})),
\end{align*}
for all $(v_{ij},v_{kl})\in E\backslash\left\{(v_{32},v_{22})\right\}$ and $0\leq\alpha\leq 4$.
\par Consider now the special edge $(v_{32},v_{22})$. Using the knowledge of $F_{\alpha}$ on the other edges, $u_{\alpha}(v_{22})$ and $u_{\alpha}(v_{32})$ are given by
\begin{align*}
  u_{\alpha}(v_{22})&=f(v_{22})+F_{\alpha}((v_{32},v_{22}))+F_{\alpha}((v_{23},v_{22}))-F_{\alpha}((v_{22},v_{12}))-F_{\alpha}((v_{22},v_{21}))\\
 &=18+F_{\alpha}((v_{32},v_{22}))+3\alpha,
\end{align*}
and
\begin{align*}
  u_{\alpha}(v_{32})&=f(v_{32})-F_{\alpha}((v_{32},v_{22}))+F_{\alpha}((v_{33},v_{32}))-F_{\alpha}((v_{32},v_{31}))\\
  &=20-F_{\alpha}((v_{32},v_{22})),
 \end{align*}
for $0\leq\alpha\leq 4$. Recall further that $u_{\alpha}$ is the $\ell^{2}$-minimizer in the set $f-\divergence\mathcal{B}_{\alpha}$, cf.~Proposition \ref{thm:rofconstrained}, and that $F_{\alpha}((v_{32},v_{22}))$ only appears in the terms $u_{\alpha}(v_{22})$ and $u_{\alpha}(v_{32})$. Minimizing $(u_{\alpha}(v_{22}))^2+(u_{\alpha}(v_{32}))^2$ subject to the constraint $F_{\alpha}((v_{32},v_{22}))\in\left[-\alpha,\alpha\right]$ then gives
\begin{align*}
 F_{\alpha}((v_{32},v_{22}))=
 \left\{
\begin{array}{rl}
\alpha, & 0\leq\alpha\leq 2/5,\\
(2-3\alpha)/2, & 2/5\leq\alpha\leq 2,\\
-\alpha, & 2\leq\alpha\leq 4.
\end{array}
\right.
\end{align*}
\par The function $F_{\alpha}$ is now determined on all edges on the interval $\alpha\in\left[0,4\right]$. The ROF minimizer $u_{\alpha}$ can then be computed according to \eqref{reprof}. The results can be seen in Figure \ref{Figure2}.

\subsubsection*{TV flow}
Recall that, according to \eqref{repflow}, the solution $u(t)$ of the TV flow problem can be represented as
\begin{align*} 
 u(t)=f+\divergence(F(t)),
\end{align*}
where $F(t) = -\int^{t}_{0}H(s)ds$ and $H(s)\in \mathcal{B}_{1,u(s)}$. In particular, $F(t)\in\mathcal{B}_{t}$. Below, $F(t)$ is computed for $t\in\left[0,4\right]$ which then enables computation of $u(t)$ on this interval.
\par We have an analogous inequality to \eqref{inrof},
\begin{align} \label{inflow}
 f(v)-\deg(v)t\leq u(t)(v)\leq f(v)+\deg(v)t
\end{align}
for all $v\in V$. Using \eqref{inflow}, we can derive that
\begin{align} \label{sgneq2}
\sgn(u(t)(v_{kl})-u(t)(v_{ij}))=\sgn(f(v_{kl})-f(v_{ij})) \in \{\pm 1\}
\end{align}
holds for any edge $(v_{ij},v_{kl})\in E\backslash\left\{(v_{32},v_{22})\right\}$ and $0\leq t\leq 4$. From \eqref{sgneq2} and $H(s)\in \mathcal{B}_{1,u(s)}$ it follows in turn that
\begin{align*}
H(s)((v_{ij},v_{kl}))=\sgn(f(v_{kl})-f(v_{ij}))
\end{align*}
for all $(v_{ij},v_{kl})\in E\backslash\left\{(v_{32},v_{22})\right\}$ and $0\leq t\leq4$. Hence,
\begin{align*}
F(t)((v_{ij},v_{kl}))=-\int^{t}_{0}H(s)((v_{ij},v_{kl}))ds=t\sgn(f(v_{ij})-f(v_{kl})),
\end{align*}
for all $(v_{ij},v_{kl})\in E\backslash\left\{(v_{32},v_{22})\right\}$ and $0\leq t\leq 4$.
\par Turn next to the computation of $F(t)((v_{32},v_{22}))$ on $0\leq t\leq 4$. Knowledge of $F(t)$ on the other edges gives
\begin{align} \label{ualpha1}
 u(t)(v_{22})=18+3t+F(t)((v_{32},v_{22})),
\end{align}
and
\begin{align} \label{ualpha2}
 u(t)(v_{32})=20-F(t)((v_{32},v_{22})),
\end{align}
on $0\leq t\leq 4$. From \eqref{ualpha1} and \eqref{ualpha2}, together with $F(t)\in\mathcal{B}_{t}$, follow the inequalities
\begin{align*}
u(t)(v_{22})\leq 18 + 4t< 20-t\leq u(t)(v_{32}), \quad 0\leq t< 2/5.
\end{align*}
These inequalities imply that
\begin{align*}
\sgn(u(t)(v_{22})-u(t)(v_{32}))=-1,\quad 0\leq t< 2/5,
\end{align*}
and therefore
\begin{align*}
H(t)((v_{32},v_{22}))=-1,\quad 0\leq t< 2/5.
\end{align*}
We then obtain
\begin{align*}
F(t)((v_{32},v_{22}))=-\int^{t}_{0}H(s)((v_{32},v_{22}))ds=t, \quad 0\leq t\leq 2/5.
\end{align*}
Consider now the interval $2/5\leq t\leq 4$ where we estimate
\begin{align*}
F(t)((v_{32},v_{22}))&=F(2/5)((v_{32},v_{22}))-\int^{t}_{2/5}H(s)((v_{32},v_{22}))ds\\
	&\geq 2/5-(t-2/5)=4/5-t
\end{align*}
This inequality together with \eqref{ualpha1} and \eqref{ualpha2} give
\begin{align*}
u(t)(v_{32})\leq 96/5+t< 94/5+2t\leq u(t)(v_{22}), \quad 2/5< t\leq 4.
\end{align*}
From these inequalities it follows that
\begin{align*}
H(t)((v_{32},v_{22}))=\sgn(u(t)(v_{22})-u(t)(v_{32}))=1, \quad 2/5<t\leq 4,
\end{align*}
which in turn gives
\begin{align*}
 F(t)((v_{32},v_{22}))=4/5-t, \quad 2/5\leq t\leq 4.
\end{align*}
\par The function $F(t)$ is now determined on all edges on the interval $t\in\left[0,4\right]$. The solution $u(t)$ of the TV flow problem can then be computed according to \eqref{repflow}. The results can be seen in Figure \ref{Figure3}.

\begin{footnotesize}

\end{footnotesize}

\end{document}